\documentclass[11pt,psamsfonts]{amsart}
\usepackage{amsmath}
\usepackage{amsthm}
\usepackage{amssymb}
\usepackage{amscd}
\usepackage{amsfonts}
\usepackage{amsbsy}
\usepackage{hyperref}
\usepackage{graphicx}

\def\be{\begin{equation}}
\def\ee{\end{equation}}
\def\bq{\begin{eqnarray}}
\def\eq{\end{eqnarray}}

\def\beq{\begin{eqnarray*}}
\def\eeq{\end{eqnarray*}}

\newtheorem {theorem} {Theorem}
\newtheorem {proposition} [theorem]{Proposition}
\newtheorem {definition} [theorem]{Definition}

\newtheorem {lemma}  [theorem]{Lemma}

\begin{document}

\title[Transversal conics and the existence of limit cycles]
{Transversal conics and the existence of limit cycles}

\author[H. Giacomini and M. Grau]
{H\'ector Giacomini$^1$ and Maite Grau$^2$}

\address{$^1$ Laboratoire de Math\'ematiques et Physique Th\'eorique.
C.N.R.S. UMR 7350., Facult\'e des Sciences et Techniques.
Universit\'e de Tours., Parc de Grandmont 37200 Tours, France.}

\email{Hector.Giacomini@lmpt.univ-tours.fr}

\address{$^2$ Departament de Matem\`atica, Universitat de Lleida,
Avda. Jaume II, 69; 25001 Lleida, Catalonia, Spain}

\email{mtgrau@matematica.udl.cat}


\date{}
\dedicatory{}

\maketitle

\begin{abstract}
This paper deals with the problem of location and existence of limit cycles for real planar polynomial differential systems. We provide a method to construct Poincar\'e--Bendixson regions by using transversal conics. We present several examples of known systems in the literature showing different features about limit cycles: hyperbolicity, Hopf bifurcation, sky-blue bifurcation, rotated vector fields, \ldots for which the obtained Poincar\'e--Bendixson region allows to locate the limit cycles. Our method gives bounds for the bifurcation values of parametrical families of planar vector fields and intervals of existence of limit cycles.
\end{abstract}

\section{Introduction\label{sect1}}

We consider real planar polynomial differential systems of the form
\begin{equation} \label{eq1} \dot{x} \, = \, P(x,y), \quad \dot{y} \, = \, Q(x,y), \end{equation}
where the dot denotes derivation with respect to an independent variable $t$ and where $P(x,y)$ and $Q(x,y)$ are real polynomials with $P(0,0) \, =\, Q(0,0) \, = \, 0$. One of the open problems related with a system (\ref{eq1}) is to determine the number and location of its limit cycles. Poincar\'e--Bendixson theorem, see for instance \cite{DuLlAr, Perko} and also Theorem \ref{thpb}, can be very useful to prove the existence of a limit cycle and to give a region where it is located. However, this result is hardly found in applications due to the difficulty of constructing the boundaries of a Poincar\'e--Bendixson region. An example of that appears in \cite{Odani97} where Poincar\'e--Bendixson regions are used to provide a bound for the amplitude of the van der Pol limit cycle. The curves in the boundary of the Poincar\'e--Bendixson region provided in \cite{Odani97} are very intricate and the proof of its transversality is quite overwhelming. Another example appears at the beginning of Chapter 4 of \cite{ZhangZF}. It is proved that  the Li\'enard differential equation $\ddot{x} + f(x) \dot{x} + g(x) \, = \, 0$ has at least one closed orbit, under certain assumptions on the functions $f(x)$ and $g(x)$. One of the tools that is used is the construction of a Poincar\'e--Bendixson region using several segments of curves as boundaries. Our aim in this work is to find transversal conics which define Poincar\'e--Bendixson regions and thus, to prove the existence of limit cycles.

\smallskip

Consider a continuous curve $C$ in $\mathbb{R}^2$. We always assume that $C$ is nonempty. A \emph{ contact point} of $C$ is a point of the curve in which the flow of system (\ref{eq1}) and the curve are tangential.
\begin{definition}
We say that $C$ is \emph{transversal} if the flow of system {\rm (\ref{eq1})} crosses it in the same direction on all its points, except maybe a finite set of contact points.
\end{definition}
The curve $C$ is given in an implicit way when $C \, = \, \{ f(x,y) \, = \, 0 \}$ where $f(x,y)$ is a real function of class $\mathcal{C}^1$ for which the number of points satisfying the following system of equations is finite:
\[ f(x,y) \, = \, 0, \quad \frac{\partial f}{\partial x}(x,y) \, = \, 0, \quad \frac{\partial f}{\partial y}(x,y) \, = \, 0. \]
This last condition is equivalent to say that $f(x,y)$ is a square-free polynomial in case that $C$ is an algebraic curve.
We note that if $C \, = \, \{ f(x,y) \, = \, 0 \}$ then the contact points of $C$ are the points in $\mathbb{R}^2$ which satisfy the following two equations $\dot{f}(x,y) \, = \, 0$ and $f(x,y)=0$, where
\[ \dot{f}(x,y) \, := \, P(x,y) \, \frac{\partial f}{\partial x}(x,y) \, + \, Q(x,y) \, \frac{\partial f}{\partial y}(x,y). \]
In this case, $C$ is a transversal curve if either $\dot{f}(x,y) \geq 0$ for all $(x,y) \in C$ or $\dot{f}(x,y) \leq 0$ for all $(x,y) \in C$, and $\dot{f}(x,y)=0$ only on a finite number of points of $C$. \par
The curve $C$ is given in terms of a parametrization when $C \, = \, \{ \zeta(\sigma) \, : \, \sigma \in \mathcal{I} \}$ where $\mathcal{I}$ is a real interval and $\zeta(\sigma) \, = \, (\zeta_1(\sigma), \zeta_2(\sigma))$ is a real function of class $\mathcal{C}^1$ with $\zeta: \mathcal{I} \subseteq \mathbb{R} \longrightarrow \mathbb{R}^2$. In this case, the contact points of $C$ are the points $\zeta(\sigma)$ for which $\sigma$ satisfies the equation $\dot{\zeta}(\sigma) \, = \, 0$ where
\[ \dot{\zeta}(\sigma) \, := \, P\left(\zeta(\sigma)\right) \left(-\zeta_2'(\sigma) \right) \, + \, Q\left(\zeta(\sigma)\right) \, \zeta_1'(\sigma). \] And $C$ is a transversal curve if
either $\dot{\zeta}(\sigma) \geq 0$ for all $\sigma \in \mathcal{I}$ or $\dot{\zeta}(\sigma) \leq 0$ for all $\sigma \in \mathcal{I}$, and $\dot{\zeta}(\sigma)=0$ only for a finite number of values $\sigma \in \mathcal{I}$. \par
A parameterized differentiable curve $\zeta: \mathcal{I} \longrightarrow \mathbb{R}^2$ is said to be regular if $\zeta'(\sigma) \neq (0,0)$ for all $\sigma \in \mathcal{I}$. A closed plane curve $C$ is a regular parameterized curve $\zeta: [a,b] \longrightarrow \mathbb{R}^2$ such that $\zeta$ and all its derivatives have the same value at $a$ and $b$. The curve is said to be simple if it has no self-intersections, that is if $\sigma_1, \sigma_2 \in [a,b)$ and $\sigma_1 \neq \sigma_2$, then $\zeta(\sigma_1) \neq \zeta(\sigma_2)$. For further information about these classical concepts, see for instance \cite{doCarmo}.

\smallskip

A limit cycle of (\ref{eq1}) is an isolated periodic orbit. A transversal section of system (\ref{eq1}) is an arc of a curve without contact points. Given a limit cycle $\Gamma$ there always exist a transversal section $\Sigma$ which can be parameterized by $r \in (-\varepsilon,\varepsilon)$ with $\varepsilon>0$ and $r=0$ corresponding to a common point between $\Gamma$ and $\Sigma$. Given $r \in (-\varepsilon,\varepsilon)$, we consider the flow of system (\ref{eq1}) with initial point the one corresponding to $r$ and we follow this flow for positive values of $t$. It can be shown, see for instance \cite{Perko}, that for $\varepsilon$ small enough, the flow cuts $\Sigma$ again at some point $\mathcal{P}(r)$. The map $r \longrightarrow \mathcal{P}(r)$ is called the Poincar\'e map associated to the limit cycle $\Gamma$ of system (\ref{eq1}). It is clear that $\mathcal{P}(0) \, = \, 0$. If $\mathcal{P}'(0) \neq 1$, the limit cycle $\Gamma$ is said to be hyperbolic. If the expansion of $\mathcal{P}(r)$ around $r=0$ is of the form $\mathcal{P}(r) \, = \, r \, + \, a_\mu r^\mu \, + \, \mathcal{O}(r^{\mu+1})$ with $a_{\mu} \neq 0$ and $\mu \geq 2$, we say that $\Gamma$ is a multiple limit cycle of multiplicity $\mu$. A classical result, see for instance \cite{Perko}, states that if $\Gamma \, = \, \left\{ \gamma(t) \, : \, t \in [0,T) \right\}$, where $\gamma(t)$ is the parametrization of the limit cycle in the time variable $t$ of system (\ref{eq1}) and $T>0$ is the period of $\Gamma$, that is, the lowest positive value for which $\gamma(0) = \gamma(T)$, and $\gamma(0) = \Gamma \cap \Sigma$, then
\begin{equation} \label{eqpp} \mathcal{P}'(0) \, = \, \exp \left\{ \int_{0}^{T} {\rm div}\left( \gamma(t) \right) dt \right\}, \end{equation}
where \[ {\rm div}(x,y) \, = \, \frac{\partial P}{\partial x}(x,y) \, + \, \frac{\partial Q}{\partial y}(x,y) \] is the {\em divergence} of the differential system (\ref{eq1}). It is clear that if $\mathcal{P}'(0)>1$ (resp. $\mathcal{P}'(0)<1$), then $\Gamma$ is an unstable (resp. stable) limit cycle. If $\Gamma$ is a multiple limit cycle of multiplicity $\mu$ and $\mu$ is odd, then $\Gamma$ is unstable if $a_\mu>0$ and stable if $a_\mu<0$. If $\mu$ is even, then the limit cycle $\Gamma$ is said to be semi-stable. For the definitions and related results, see for instance \cite{DuLlAr, Perko, YanQian}.
\par
The Poincar\'e--Bendixson theorem, which can be found for instance in section 1.7 of \cite{DuLlAr} or in section 3.7 of \cite{Perko}, has as a corollary the following result which motivates the definition of Poincar\'e--Bendixson region. See also Theorem 4.7 of Chapter 1 of \cite{ZhangZF}.
\begin{theorem} \label{thpb}
Suppose $R$ is a finite region of the plane $\mathbb{R}^2$ lying between two simple closed curves $C_1$ and $C_2$. If
\begin{itemize}
\item[{\rm (i)}] the curves $C_1$ and $C_2$ are transversal for system {\rm (\ref{eq1})} and the flow crosses them towards the interior of $R$, and
\item[{\rm (ii)}] $R$ contains no critical points.
\end{itemize}
Then, system {\rm (\ref{eq1})} has an odd number of limit cycles (counted with multiplicity) lying inside $R$.
\end{theorem}

In such a case, we say that $R$ is a \emph{Poincar\'e--Bendixson region} for system (\ref{eq1}).

\section{Description of the method\label{sect2}}

In order to construct Poincar\'e--Bendixson regions, we will use curves formed by transversal conics. We consider conics $\{ f(x,y)\, = \, 0 \}$ of the form
\begin{equation} \label{eqf} f(x,y) \, = \, 1 + s_1 x + s_2 y + s_3 x^2 + s_4 xy + s_5 y^2, \end{equation}
with $s_i \in \mathbb{R}$, $i=\overline{1,5}$. If the system is symmetric with respect to the origin, that is invariant under the change $(x,y) \to (-x,-y)$, one can take $s_1\, =\, s_2 \, = \, 0$. \par
Given a system (\ref{eq1}), the problem of determining the values of $s_i$ for which the conic $\{ f(x,y) \, = \, 0 \}$ is transversal to the flow (if these values exist) is a complicated problem due to the overwhelming calculations which deal with the five parameters $s_i$, $i=\overline{1,5}$. We describe a method in which we consider one-parameter conics and then, the determination of for which values (if exist) it is transversal is much easier. We present the method in six steps.

\subsection{First step: local solution of the system}

Fix a natural $N \geq 4$ and consider
\[ \tilde{x}(t) \, = \, a_0 + \sum_{i=1}^{N} a_i t^i, \quad \tilde{y}(t) \, = \, b_0 + \sum_{i=1}^{N} b_i t^i, \] with $a_i, b_i \in \mathbb{R}$. Find $a_i, b_i$, for $i=1,2,\ldots,N$, so that $(\tilde{x}(t),\tilde{y}(t))$ is an approximation of the solution of system (\ref{eq1}) with initial condition $(a_0,b_0)$, that is
\[ \dot{\tilde{x}}(t) - P\left(\tilde{x}(t), \tilde{y}(t)\right) \, = \, \mathcal{O}(t^N), \quad \dot{\tilde{y}}(t) - Q\left(\tilde{x}(t), \tilde{y}(t)\right) \, = \, \mathcal{O}(t^N). \] Note that the initial condition $(a_0,b_0)$ is not fixed. We take, for instance, either $b_0=0$ or $a_0=0$ and then we get an approximation  for small values of $|t|$ of the solution up to order $N$ in the independent variable $t$ and at the point $(a_0,0)$ or $(0,b_0)$. As we will see in the examples, in certain cases it is better to choose a point of the form $(a_0,0)$ as initial condition and in other cases it is better to choose a point of the form $(0,b_0)$, in order to find a transversal conic.

\subsection{Second step: glue the curve $f=0$ to the solution}

Find the values $s_i$, $i=\overline{1,5}$, in (\ref{eqf}) so that the Taylor expansion of $f\left(\tilde{x}(t),\tilde{y}(t)\right)$ at $t=0$ has an order as high as possible. In general $ \displaystyle f\left(\tilde{x}(t),\tilde{y}(t)\right) \, = \, \mathcal{O}(t^5)$. In this way we get a single parameter conic which only depends on $a_0$ or $b_0$.

\subsection{Third step: parameterize the curve $f=0$}

Find a rational pa\-ra\-me\-tri\-za\-tion $\pi: D \subseteq \mathbb{R} \to \mathbb{R}^2$ of the conic $f(x,y)=0$ such that the parameter is uniform, that is $\pi'(m) \neq (0,0)$ for all $m \in D$, see for instance Chapter 3 of \cite{Walker}. The domain $D$ of definition of the parametrization of a conic can be chosen to be $\mathbb{R}$ minus $0$, $1$ or $2$ points.  Define the rational function
\[ \varphi(m) \, := \, P(\pi(m)) \, \frac{\partial f}{\partial x}(\pi(m)) \, + \, Q(\pi(m)) \, \frac{\partial f}{\partial y}(\pi(m)). \]
Remark that the denominator of $\varphi(m)$ is a power of the denominator of $\pi(m)$. The zeros of the denominator of $\pi(m)$ correspond to the ``points at infinity'' of the conic $f=0$ and the denominator of $\pi(m)$ has no real roots if the conic is an ellipse; in general, has one (double) real root if the conic is a parabola; or, in general, has two simple real roots if the conic is a hyperbola. We can check the type of the conic by looking at the denominator of $\pi(m)$.

\smallskip

The real plane $\mathbb{R}^2 = \{(x,y) \, : \, x,y \in \mathbb{R} \}$ can be embedded in the sphere $\mathbb{S}^2 \, = \, \{ (X,Y,Z) \in \mathbb{R}^3 \, : \, X^2+Y^2+Z^2=1\}$ by means of the Poincar\'e compactification, see for instance Chapter 5 of \cite{DuLlAr}. The plane is embedded in the northern hemisphere $\{ (X,Y,Z) \in \mathbb{S}^2 \, : \, Z>0\}$ by the projection $p^{+}$ and also in the southern hemisphere $\{ (X,Y,Z) \in \mathbb{S}^2 \, : \, Z<0\}$ by the projection $p^{-}$ where
\[ \begin{array}{llll} p^{\pm} : & \mathbb{R}^2 & \longrightarrow & \mathbb{S}^2 \\
& (x,y) & \mapsto & \displaystyle \left( \frac{\pm\, x}{\sqrt{x^2+y^2+1}}, \, \frac{\pm\, y}{\sqrt{x^2+y^2+1}}, \, \frac{\pm \, 1}{\sqrt{x^2+y^2+1}} \right).  \end{array} \] We consider one of the hemispheres together with the equator $\{(X,Y,Z) \, : \, Z=0 \}$ and we project them in the disk $\{ (x,y) \in \mathbb{R}^2 \, : \, x^2+y^2 \leq 1 \}$, which we denote as the Poincar\'e disk. \par Given a real algebraic curve $f(x,y)=0$ of degree $n$, we can embed it in the sphere $\mathbb{S}^2$ as the curve $F(X,Y,Z)=0$ by \[ F(X,Y,Z) \, := \, Z^n \, f\left(\frac{X}{Z}, \, \frac{Y}{Z} \right). \] The ``points at infinity'' of $f(x,y)=0$ correspond to the intersection points of the curve on the sphere $F(X,Y,Z)=0$ with the equator, that is, the points $(X,Y,0) \in \mathbb{S}^2$ such that $F(X,Y,0) \, = \, 0$. \par Analogously, the flow of system (\ref{eq1}) can be embedded into the sphere by the Poincar\'e compactification. The singular points ``at infinity'' of system (\ref{eq1}) correspond to the singular points of the embedded flow on the equator of the Poincar\'e disk.

\smallskip

Remark that if $\varphi(m) \geq 0$ or $\varphi(m) \leq 0$ for all $m \in D$, and $\varphi(m)=0$ only for a finite number of values $m$, then $f=0$ is a transversal conic for system (\ref{eq1}). We only need to check the changes of sign of the numerator of $\varphi(m)$ in order to show whether $f(x,y)=0$ is a transversal conic. Then, we check the denominator to see the type of the conic we are dealing with. \par
Remark that $\varphi(m)$ depends rationally on the initial condition $(a_0,b_0)$.

\subsection{Fourth step: necessary condition for transversal conics}

Recall that we only need to consider the numerator of $\varphi(m)$ to study its changes of sign. Suppose the initial condition is $(a_0,0)$ and that $a_0^*$ is a value for which $\varphi(m)>0$ for all $m \in D$. Then, for any $a_0$ close enough to $a_0^*$ we have that $\varphi(m)>0$ for all $m \in D$. The value $\mathbf{\tilde{a}_0}$ at the end of this interval in $a_0$ is such that $\varphi(m)$ has all its real roots with even multiplicity or may be a root of the coefficient of the monomial of highest order of the numerator of $\varphi(m)$. \par
Take the numerator of $\varphi(m)$ and consider the resultant of it with its derivative with $m$ as variable. We get a polynomial in $a_0$, which we denote by $R(a_0)$. For the definition and properties of the resultant of two polynomials see Chapter 1 of \cite{Walker}.\par We have that either all the constructed conics are transversal, either there are none or a necessary condition to have transversal conics is that the polynomial $R(a_0)$ has a real root $\mathbf{\tilde{a}_0}$.

\subsection{Fifth step: find transversal conics}

Take each one of the real roots of $R(a_0)$ and check whether all the real roots of the numerator of $\varphi(m)$ are of even multiplicity. If $\mathbf{\tilde{a}_0}$ is such a root, then it defines a transversal conic and, in general, the border of a band of transversal conics. \par We take a rational value $a_0^{*} \in \mathbb{Q}$ close to the real root $\mathbf{\tilde{a}_0}$ and such that we can prove that the numerator of the rational function $\varphi(m)$ does not change sign using Sturm sequences. \par In this way, we have reduced the problem of finding transversal curves to one which is purely algebraic. This problem can be very difficult to solve when the system has parameters. In the latter case, and in order to apply Sturm sequences, we need to study the sign of expressions which depend on the parameters of the system. This is a sharp problem which we only tackle in Examples 1 (subsection \ref{sect31}) and 2 (subsection \ref{sect32}). In the other examples either the system has no parameters or the system is a (semi-)complete family of rotated vector fields with respect to one parameter, see for instance section 4.6 in \cite{Perko} or section 4.3 in \cite{ZhangZF} for a definition. The knowledge of the properties of the bifurcation diagram of limit cycles in rotated vector fields allows us to give Poincar\'e--Bendixson regions for fixed values of the parameter and establish results for values of the parameter in a certain interval.

\subsection{Sixth step: Poincar\'e--Bendixson regions}

Two nested transversal ellipses whose corresponding $\varphi(m)$ is of opposite sign define a region in which one can directly apply Theorem \ref{thpb} and prove the existence of at least one limit cycle, provided that the region contains no critical points. A transversal hyperbola or a transversal parabola together with a part of the equator of the Poincar\'e disk can define a Poincar\'e--Bendixson region. In Examples 4 (subsection \ref{sect34}) and 5 (subsection \ref{sect35}) we make use of the latter type of Poincar\'e--Bendixson regions.

\section{Examples\label{sect3}}

We present six different examples in which a Poincar\'e--Bendixson region can be constructed by using the method described above.

\subsection{Example 1: a quadratic system\label{sect31}}

The following system appears in p. 347 of \cite{YanQian}.
\begin{equation} \label{eqex1} \dot{x} \, = \, \delta x + \ell x^2 -y, \quad \dot{y} \, = \, x(1+ax+by), \end{equation} where $a,b, \delta, \ell \in \mathbb{R}$ and with $\delta a(b+2 \ell)>0$. In \cite{chen77} it is proved that this system has a unique limit cycle for $\delta \in (0,\delta^*)$ or $\delta \in (\delta^*,0)$. The value $\delta^*$ is a \emph{bifurcation value} for which the limit cycles disappears. We will fix three of the four parameters and we try to study which is the behavior of the limit cycle for $\delta$ close to $\delta^*$.

\subsubsection{System {\rm (\ref{eqex1})} with $a \, =\, b\, =\,  \ell\, =\, 1$.}

We take
\begin{equation} \label{eqex11} \dot{x} \, = \, \delta x + x^2-y, \quad \dot{y} \, = \, x(1+x+y) \end{equation}
and $\delta>0$. \par Note that when $\delta<1$, the only finite singular point is the origin, which is an unstable focus. For each $\delta \in (0,1)$, we find a Poincar\'e--Bendixson region for system (\ref{eqex11}) bounded by two transversal ellipses, one in the interior of the region bounded by the limit cycle and the other in its exterior. Thus, we prove the existence of at least one limit cycle for $\delta \in (0,1)$. We take the initial condition $(0,b_0)$ in this case. The expression of the polynomial $f(x,y)$ corresponding to the application of the first and second steps described in section \ref{sect2} is:
\[ \begin{array}{lll} f(x,y) & = & \displaystyle 1 \, +\, \Big[  9 (b_0+1)^4 x^2 -12 (b_0+1)^2 (b_0 \delta+b_0+\delta) x y + \big(18 b_0^3+ \vspace{0.2cm} \\
& & \displaystyle 2 b_0^2 \delta^2-14 b_0^2 \delta+29  b_0^2+4 b_0 \delta^2-14 b_0 \delta+36 b_0+2 \delta^2+9\big) y^2 +
\vspace{0.2cm} \\ & & \displaystyle
 12 b_0 (b_0+1)^2 (b_0 \delta+b_0+\delta) x -2 b_0 \big(9 b_0^3+2 b_0^2 \delta^2-14
   b_0^2 \delta+
\vspace{0.2cm} \\ & & \displaystyle
2 b_0^2+4 b_0 \delta^2-14 b_0 \delta+9 b_0+2 \delta^2\big) y
\Big] \bigg / \vspace{0.2cm} \\ & & \displaystyle \Big[b_0^2 \big(2 b_0^2 \delta^2-14 b_0^2 \delta-25
   b_0^2+4 b_0 \delta^2-14 b_0 \delta  -18 b_0+2 \delta^2-9\big)\Big].
\end{array} \]

The expression of the function $\varphi(m)$ defined in the third step of section \ref{sect2} is of the form \begin{equation} \label{eqfi1} \varphi(m) \, = \, m^4 \, \frac{P_2(m,b_0,\delta)}{\tilde{P_2}(m,b_0,\delta)^3}, \end{equation} where $P_2$ and $\tilde{P_2}$ are polynomials in $m$, $b_0$ and $\delta$ of degree $2$ in $m$. \par

We denote by $b_{0i}(\delta)$ the value corresponding to the interior ellipse and by $b_{0e}(\delta)$ the one of the exterior ellipse.
In the following two Figures \ref{casocuadratico_d1s10} and \ref{casocuadratico_d9s10}, we represent in the $(x,y)$-plane the limit cycle (numerically found) in red, and the two transversal ellipses in blue for $\delta \, = \, 1/10$ and $\delta \, = \, 9/10$, respectively.

\begin{figure}[htb]
\includegraphics[width=0.5\textwidth]{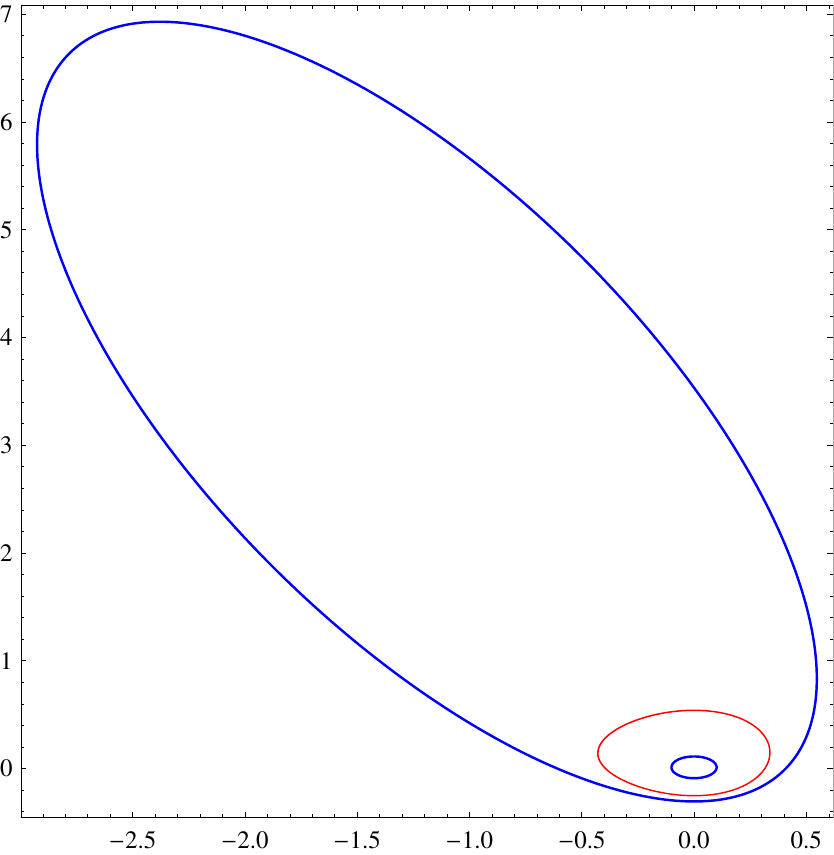}
\caption{Stable limit cycle in red and transversal ellipses in blue when $\delta \, = \, 1/10$. \[  b_{0i} \, = \, -\frac{1}{11}\, \approx \, -0.091, \quad \ b_{0e}\, = \,-\frac{613}{2000}\, \approx \, -0.306. \]}
\label{casocuadratico_d1s10}
\end{figure}

\begin{figure}[htb]
\includegraphics[width=0.5\textwidth]{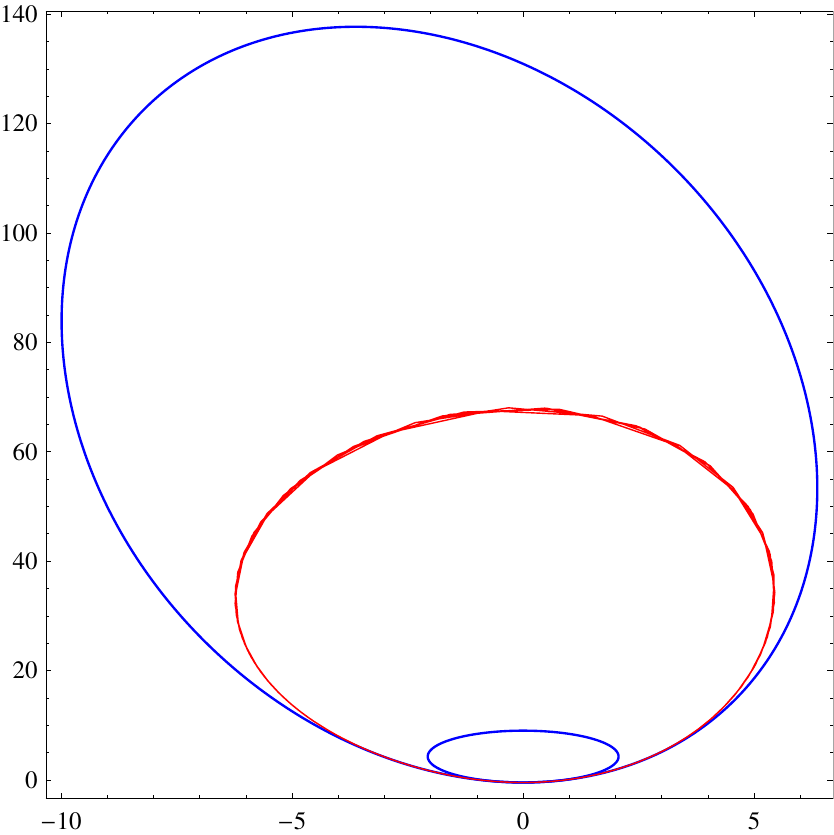}
\caption{Stable limit cycle in red and transversal ellipses in blue when $\delta \, = \, 9/10$. \[  b_{0i} \, = \, -\frac{9}{19}\, \approx \, -0.474, \quad b_{0e} \, = \,-\frac{1646850590977}{3435973836800}\, \approx \, -0.479. \]}
\label{casocuadratico_d9s10}
\end{figure}

For any $\delta \in (0,1)$, the interior ellipse is a circle and corresponds to \[ b_{0i}(\delta) \, = \, -\, \frac{\delta}{\delta+1}. \] If we substitute this expression in (\ref{eqfi1}) we get
\[ \varphi(m) \, = \, \frac{16(\delta-\delta^3)m^4}{\left( (\delta^2-1)m^2-1\right)^3}. \] We observe that since $\delta \in (0,1)$ we have that $\varphi(m) < 0$ for all $m \neq 0$ and $\varphi(0)=0$. Thus, the curve that we have obtained is transversal. Indeed, since $(\delta^2-1)m^2-1$ has no real roots in $m$ when $0<\delta<1$, we have that the curve is an ellipse. \par The exterior ellipse corresponds to
\[ b_{0e}(\delta) \, = \, \left\{ \begin{array}{ll} \displaystyle (-43 \delta-57)/200 & \mbox{if } \displaystyle \delta \, \leq \, \frac{647}{1000}, \\ \vspace{0.2cm}
 -\frac{1}{2} -\frac{13(\delta-1)}{64}+\frac{1339(\delta-1)^2}{32768} & \\ +\frac{137467
   (\delta-1)^3}{8388608}-\frac{115697945 (\delta-1)^4}{8589934592}& \mbox{if } \displaystyle \delta \,  > \, \frac{647}{1000}. \end{array} \right. \]
The function $b_{0e}(\delta)$ is not continuous in $\delta=0.647$ because \[ \lim_{\delta \to 0.647^{-}} b_{0e}(\delta) \, = \, 0.4241 \ldots \quad \mbox{and} \quad \lim_{\delta \to 0.647^{+}} b_{0e}(\delta) \, = \, 0.4256 \ldots. \] In order to prove whether this expression gives rise to a transversal ellipse for the flow of system (\ref{eqex11}), we need to study the changes of sign of $\varphi(m)$ given in (\ref{eqfi1}), once we have substituted $b_0$ by $b_{0e}(\delta)$. \par If
\[ b_{0e}(\delta) \, = \, \frac{-43\delta - 57}{200}, \] then the discriminant with respect to $m$ of the polynomial $P_2$ defined in (\ref{eqfi1}) is $(\delta-1)(43 \delta - 143)^4 (43 \delta - 57) Q_{12}(\delta)$ where $Q_{12}(\delta)$ is a polynomial of degree $12 $ in $\delta$ which has no real roots for $\delta \in (0,0.647]$, which can be proved using Sturm sequences. And the discriminant of the polynomial $\tilde{P_2}$ defined in (\ref{eqfi1}) is $(\delta-1)(43 \delta-143)^4Q_3(\delta)$ where $Q_3(\delta)$ is a polynomial of degree $3$ in $\delta$ which has no real roots for $\delta \in (0,0.647]$. Since these discriminants are negative for $\delta \in (0,0.647]$, we have that $\varphi(m)$ does not change sign and that the curve $f=0$ is an ellipse. \par If
\[ \begin{array}{lll} \displaystyle b_{0e}(\delta) & = & \displaystyle -\frac{1}{2} -\frac{13(\delta-1)}{64}+\frac{1339(\delta-1)^2}{32768} \vspace{0.2cm} \\ & & \displaystyle +\frac{137467 (\delta-1)^3}{8388608}-\frac{115697945 (\delta-1)^4}{8589934592}, \end{array} \]
then the discriminant with respect to $m$ of the polynomial $P_2$ defined in (\ref{eqfi1}) is of the form $(\delta-1)^7 Q_4(\delta)^{14} Q_{32} (\delta)$ where $Q_j(\delta)$ is a polynomial of degree $j$ in $\delta$ without real roots in the interval $\delta \in (0.647,1)$, $j=4,32$. And the discriminant with respect to $m$ of the polynomial $\tilde{P_2}$ defined in (\ref{eqfi1}) is of the form $(\delta-1) \tilde{Q}_4(\delta)^4 Q_{11}(\delta)$ where $\tilde{Q}_4(\delta)$ and $Q_{11}(\delta)$ are polynomials in $\delta$ of degrees $4$ and $11$, respectively, without real roots in the interval $\delta \in (0.647,1)$. Again, since these discriminants are negative for $\delta \in (0.647,1)$, we have that $\varphi(m)$ does not change sign and that the curve $f=0$ is an ellipse.

\smallskip

In Figure \ref{casocuadratico_b0ib0e} we represent the functions $b_{0i}(\delta)$ and $b_{0e}(\delta)$ in the $(\delta,b_0)$-plane, for $\delta \in (0,1)$.
\begin{figure}[htb]
\includegraphics[width=0.5\textwidth]{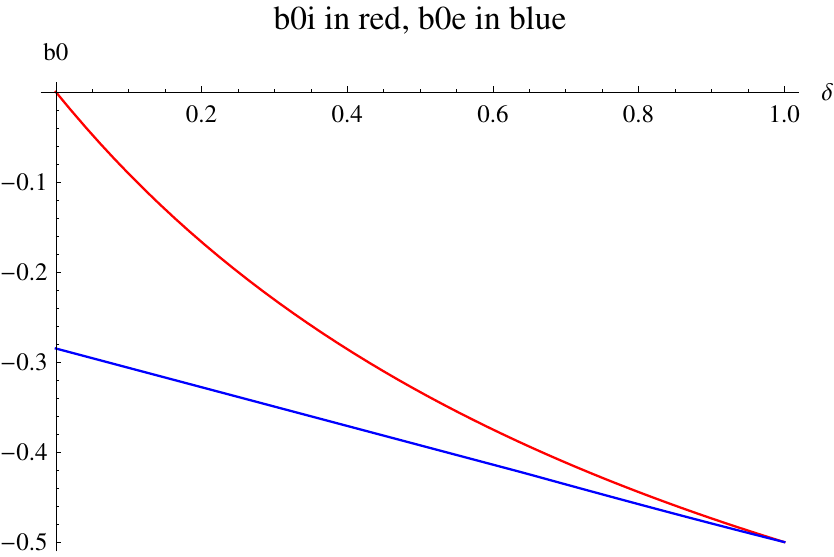}
\caption{Values $b_{0i}(\delta)$ and $b_{0e}(\delta)$ for the transversal ellipses.}
\label{casocuadratico_b0ib0e}
\end{figure}

We note that \[ b_{0i}(\delta) - b_{0e}(\delta) \, = \, \frac{3}{64}(1-\delta) + \mathcal{O}(1-\delta)^2. \] There is a band of exterior ellipses for a certain region of the space of parameters $(\delta,b_0)$. We take a piecewise rational function in the space of parameters $(\delta,b_0)$ inside this region in order to find this $b_{0e}(\delta)$ for the exterior ellipse. When $\delta \to 1$, the band becomes narrower and we have approximated the limiting curves by their Taylor expansion when $\delta \to 1$. We remark that both the inner and the outer transversal ellipses coalesce in the parabola $1 - x^2 + 2 y=0$ when $\delta=1$.

\smallskip

Given an algebraic curve $f(x,y)=0$ where $f(x,y)$ is a real polynomial, we say that it is an {\em invariant algebraic curve} for system (\ref{eq1}) if there exists a polynomial $k(x,y)$ such that
\[ P(x,y) \, \frac{\partial f}{\partial x}(x,y) \, + \, Q(x,y) \, \frac{\partial f}{\partial y}(x,y) \, = \, k(x,y) \, f(x,y), \] for all $(x,y) \in \mathbb{R}^2$. We always assume that the curve $f=0$ is not empty and that $f(x,y)$ is square-free. The function $k(x,y)$ is called the cofactor of $f$ and it is a polynomial of degree lower than the degree of $P(x,y)$ and $Q(x,y)$. Note that if $f(x,y)=0$ is an invariant algebraic curve of the differential system (\ref{eq1}), then it is formed by orbits of the system. See, for instance, \cite{CGG05} and references therein for more information about invariant algebraic curves. \par
Let $\Gamma$ be a limit cycle of system (\ref{eq1}) whose parametrization using the independent variable $t$ is $\Gamma \, = \, \left\{ \gamma(t) \, : \, t \in [0,T) \right\}$ where $T>0$ is the period of $\Gamma$. And let $f(x,y)=0$ be an invariant algebraic curve of system (\ref{eq1}) which does not contain the limit cycle $\Gamma$. We can assume without loss of generality that $f(\gamma(t)) >0$ for all $t \in [0,T]$. If not, we can consider $-f$ instead of $f$. We clearly have that
\[ \begin{array}{l} \displaystyle \int_{0}^{T} \left. \frac{ P(x,y) \, \frac{\partial f}{\partial x}(x,y) \, + \, Q(x,y) \, \frac{\partial f}{\partial y}(x,y)}{f(x,y)} \right|_{(x,y) = \gamma(t)} dt \, =  \vspace{0.2cm} \\ \displaystyle \qquad = \, \left[ \ln f(\gamma(t)) \right]_{t=0}^{t=T} \, = \, \ln f(\gamma(T)) \, - \, \ln f(\gamma(0)) \, = \, 0, \end{array} \] because $\gamma(0) \, = \, \gamma(T)$. Hence,
\begin{equation} \label{eqcof}
\int_{0}^{T} k(\gamma(t)) \, dt \, = \, 0, \end{equation} where $k(x,y)$ is the cofactor associated to $f$. See \cite{GG05} for other related properties of cofactors and their relation with limit cycles.

\begin{lemma} \label{lem1}
When $\delta \, = \, 1$, the limit cycle of system {\rm (\ref{eqex11})} disappears.
\end{lemma}

\begin{proof}
When $\delta \, = \, 1$, the parabola $1-x^2+2y \, = \, 0$ is an invariant algebraic curve of system (\ref{eqex11}) with cofactor $2x$. Assume that there exists a limit cycle $\Gamma$, assume that $\Gamma$ is the closest limit cycle to the origin of coordinates and that it is parameterized using the independent variable $t$ by $\Gamma \, = \, \left\{ \gamma(t)=(\gamma_1(t), \gamma_2(t)) \, : \, t \in [0,T) \right\}$ where $T>0$ is the period of $\Gamma$. It needs to be stable since the origin is unstable. Indeed, $\Gamma$ cannot cut the parabola $1-x^2+2y \, = \, 0$ and, therefore, by (\ref{eqcof}) $\int_0^T  2 \, \gamma_1(t) \, dt \, = \, 0$. On the other hand, the divergence of system (\ref{eqex11}) with $\delta \, = \, 1$ is $\delta + 3x \, = \, 1 + 3x$. We get that
\[ \int_0^{T}  \left(1 + 3 \, \gamma_1(t)\right) \, dt \, = \, \int_0^{T}  1 \, dt \, = \, T \, > \, 0. \] This would give an unstable limit cycle from (\ref{eqpp}) which is a contradiction.  \end{proof}

\smallskip

The notion of (semi-)complete family of rotated vector fields, see for instance section 4.6 in \cite{Perko} or section 4.3 in \cite{ZhangZF}, provides in many cases the bifurcation diagram of the appearance and disappearance of a limit cycle in single parameter family of planar differential systems. Remark that the parametric family (\ref{eqex11}) is not a (semi-)complete family of rotated vector fields in $\delta$. However, as a consequence of the existence of the Poincar\'e--Bendixson regions that we have encountered for $\delta \in (0,1)$ together with the result in \cite{chen77} and Lemma \ref{lem1}, we can provide the following result.

\begin{proposition} \label{prop1}
System {\rm (\ref{eqex11})} has a unique limit cycle for $\delta \in (0,1)$. This limit cycle appears from a Hopf bifurcation when $\delta = 0$, its size changes with $\delta$ and it disappears in the parabola $1-x^2+2y \, = \, 0$ when $\delta = 1$.
\end{proposition}

Our method works for all the values of $\delta$ in which the limit cycle exists and provides the \emph{bifurcation value} $\delta \, = \, 1$.

\subsubsection{System {\rm (\ref{eqex1})} with $a \, =\,1/3$, $b\, =\, 1/5$ and $\ell\, =\, 1/7$. }

We consider the quadratic system (\ref{eqex1}) with $\displaystyle a \, = \, 1/3$, $\displaystyle b \, = \, 1/5$ and $\displaystyle \ell \, = \, 1/7$, that is
\begin{equation} \label{eqex12} \dot{x} \, = \, \delta x\, + \, \frac{1}{7}\, x^2\, -\, y, \quad \dot{y} \, = \, x \left( 1\, +\, \frac{1}{3}\, x\, +\, \frac{1}{5}\, y \right), \end{equation}
with $\delta>0$. We observe the following properties for system (\ref{eqex12}). For $\delta>0$, the origin is an unstable focus. Define $\delta_1 \, = \, \displaystyle 2 \sqrt{5/7}\, -\, 5/3 \, \approx \, 0.0236$. For $\delta \in (0,\delta_1)$, the origin is the only finite singular point. For $\delta>\delta_1$, there are two other singular points in the finite plane: a stable node and a saddle. As before, in \cite{chen77} it is proved that this system has a unique limit cycle for $\delta \in (0,\delta^*)$. Numerically, $\delta^{*} \, \approx \, 0.2465$.

\smallskip

Figures \ref{cuadraticod02} and \ref{cuadraticod03} represent the phase portrait of system (\ref{eqex12}) in the Poincar\'e disk, done with program P4 \cite{P4}, for $\delta \, = \, 0.2$ and $\delta \, = \, 0.3$. The bounding black circle of each figure represents the equator of the Poincar\'e disk and the area inside is the phase portrait in the real affine $(x,y)$-plane. Stable (resp. unstable) separatrices appear in blue (resp. red) lines. A green square represents a non-degenerate saddle, a blue (resp. red) square a non-degenerate stable (resp. unstable) node. A blue (resp. red) diamond represents a non-degenerate stable (resp. unstable) strong focus. A green triangle represents a semi-hyperbolic saddle and a blue (resp. red) triangle represents a semi-hyperbolic stable (resp. unstable) node. A black cross represents a non-elementary singular point.

\begin{figure}[htb]
\includegraphics[width=0.5\textwidth]{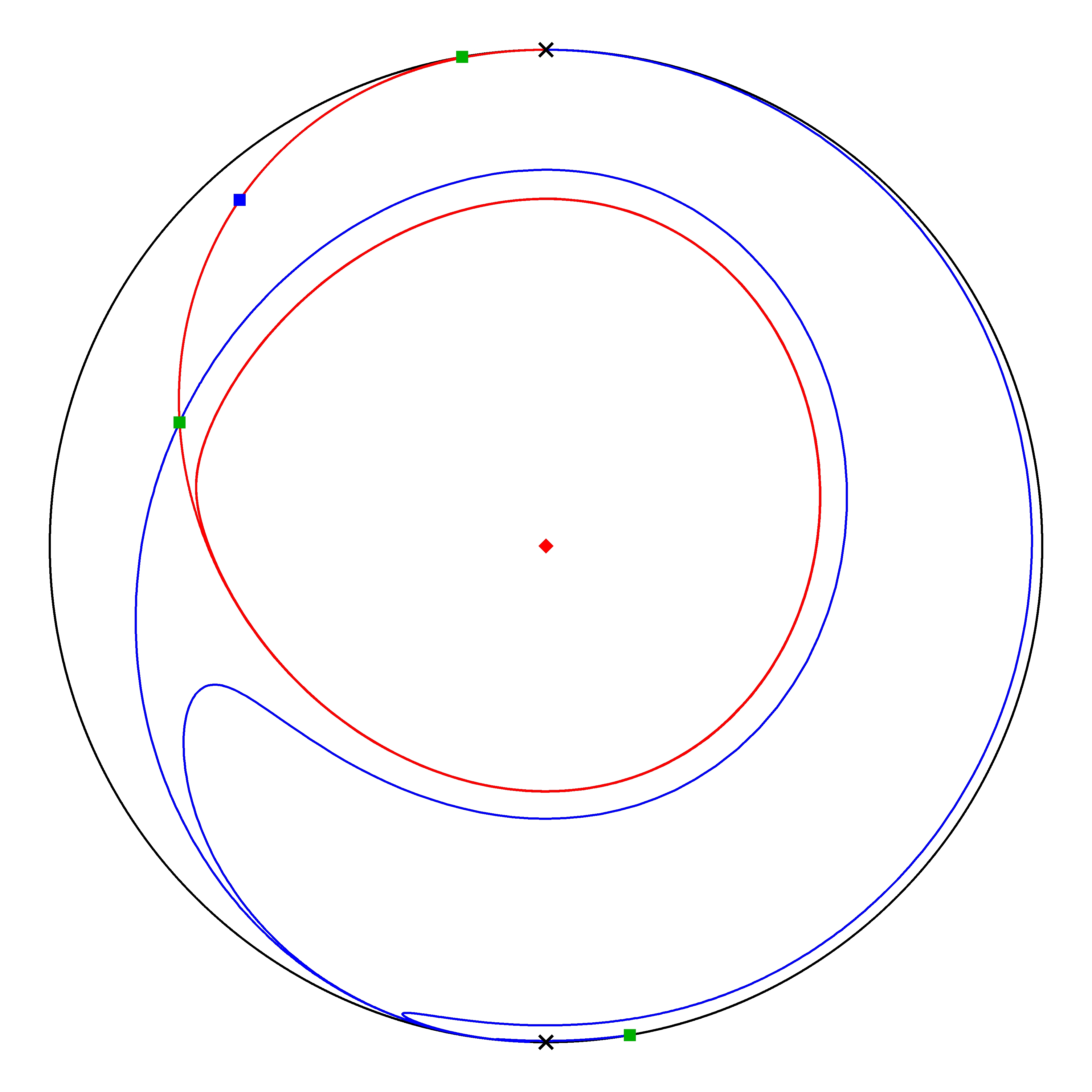}
\caption{Phase portrait in the Poincar\'e disk of system {\rm (\ref{eqex12})} for $\delta \, = \, 0.2$. There is a limit cycle surrounding the origin of coordinates.}
\label{cuadraticod02}
\end{figure}

\begin{figure}[htb]
\includegraphics[width=0.5\textwidth]{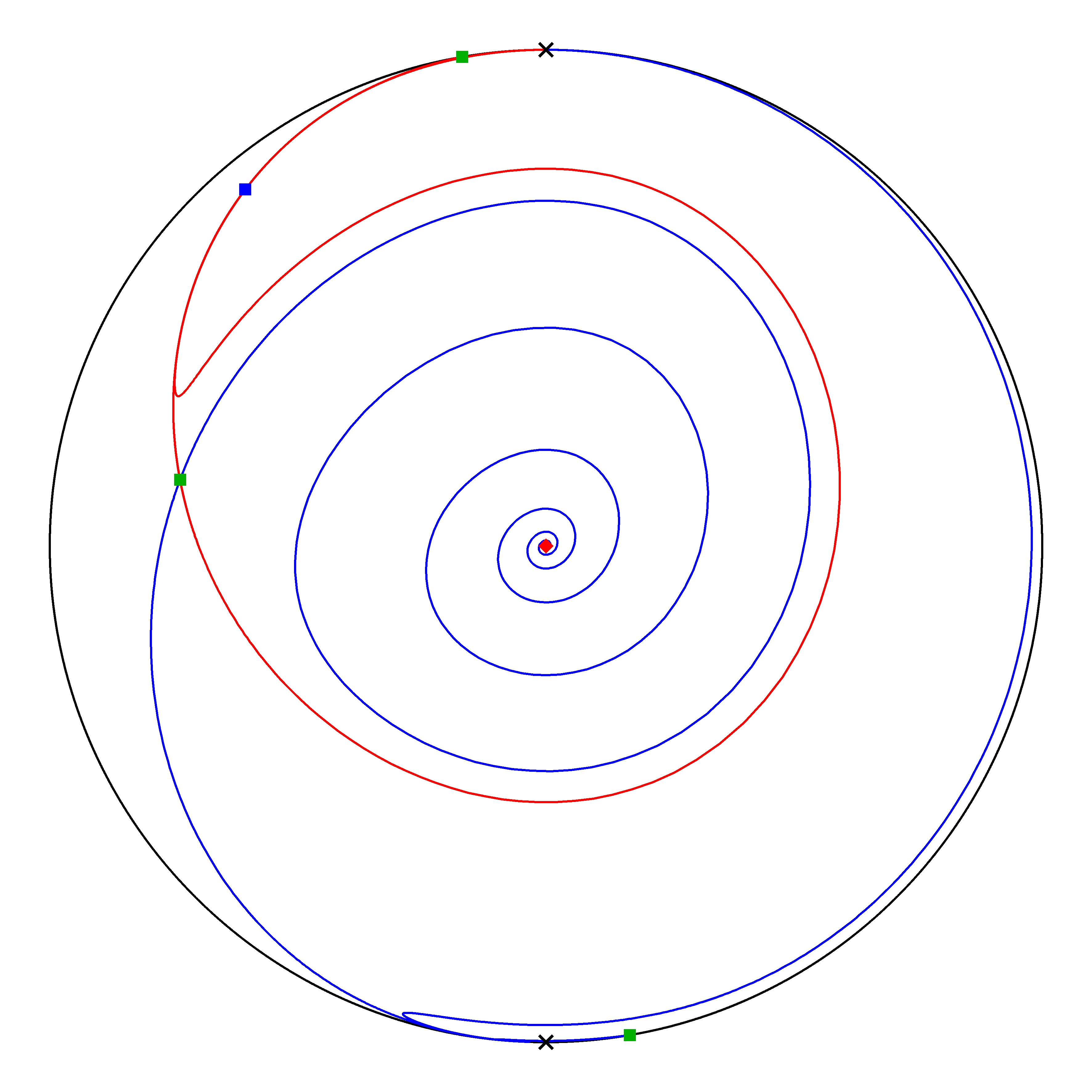}
\caption{Phase portrait in the Poincar\'e disk of system {\rm (\ref{eqex12})} for $\delta \, = \, 0.3$. There is no limit cycle but a non-oriented polycycle, surrounding the origin of coordinates.}
\label{cuadraticod03}
\end{figure}

We focus on the case in which there are three finite singular points, that is $\delta \, >\,  \delta_1$ in system (\ref{eqex12}).  By using the method described in section \ref{sect2}, we can find two transversal ellipses, one which only surrounds the unstable focus at the origin, and another one which surrounds the three singular points. These two transversal ellipses are crossed by the flow in opposite directions. Thus, by Poincar\'e--Bendixson theorem, we deduce the existence of at least one limit cycle or the existence of a non-oriented polycycle formed by the saddle whose unstable separatrices both connect with the stable node. We observe that we cannot conclude the existence of a limit cycle because inside the considered Poincar\'e--Bendixson region there are singular points. However, the existence of this Poincar\'e--Bendixson region also gives some information about the qualitative behavior of the flow of system (\ref{eqex12}).

\subsection{Example 2: van der Pol system\label{sect32}}

We consider the van der Pol system
\begin{equation} \label{eqex2} \dot{x} \, = \, y-\varepsilon \left( \frac{x^3}{3} -x \right), \quad \dot{y} \, = \, -x, \end{equation}
with $\varepsilon >0$.

The origin is the only finite critical point of the system and it is an unstable focus. It is known, see for instance \cite{Perko}, that system (\ref{eqex2}) has a unique stable and hyperbolic limit cycle for all $\varepsilon>0$ which bifurcates from the circle of radius $2$ when $\varepsilon = 0$ and which disappears to a slow-fast periodic limit set when $\varepsilon \to +\infty$. For each $\varepsilon>0$, we can find a Poincar\'e--Bendixson region bounded by two transversal ellipses, one in the interior of the region bounded by the limit cycle and the other one in its exterior. In this case we take as initial condition $(a_0,0)$ and a conic $f=0$ of the form $f(x,y)\, =\, 1 + s_3 x^2 + s_4 xy + s_5 y^2$ because the system is symmetric with respect to the origin. The expression of $f(x,y)$ that we obtain after the first and second steps described in section \ref{sect2} is
\[ f(x,y) \, = \, 1 \, -\, \frac{x^2}{a_0^2} \, - \, \frac{ 2 (3 - a_0^2) \varepsilon x y}{3 a_0^2} \, + \, \frac{(-9 - 9 \varepsilon^2 + a_0^4 \varepsilon^2) y^2}{9 a_0^2}. \]
The expression of the function $\varphi(m)$ defined in the third step of section \ref{sect2} is of the form \begin{equation} \label{eqfi2} \varphi(m) \, = \, m^2 \, \frac{P_1(m,a_0,\varepsilon)^2 P_4(m,a_0,\varepsilon)}{P_2(m,a_0,\varepsilon)^4}, \end{equation} where $P_1$, $P_4$ and $P_2$ are polynomials in $m$, $a_0$ and $\varepsilon$ of degrees $1$, $4$ and $2$, respectively, in $m$.
As we did in the previous example, we need to study the changes of sign of $\varphi(m)$ in order to prove that the curve $f=0$ is transversal to the flow of system (\ref{eqex2}). In this case, the discriminant with respect to $m$ of the polynomial $P_4$ defined in (\ref{eqfi2}) is of the form
\[ a_0^4(a_0^2-3)^2(\varepsilon^2 a_0^4-9\varepsilon^2-9)^5 Q_4(a_0,\varepsilon)^3 Q_6(a_0,\varepsilon) Q_{14}(a_0,\varepsilon)^2, \]
where $Q_j(a_0,\varepsilon)$ is a polynomial in $a_0$ and $\varepsilon$ of degree $j$ in $a_0$, $j=4,6,14$.

\smallskip

The interior ellipse is the circle of radius $\sqrt{3}$, for all $\varepsilon>0$. Thus, $a_{0i} \, = \, \sqrt{3}$. We denote by $a_{0e}$ the corresponding value for the exterior ellipse. Figures \ref{vanderPol_e1s10}, \ref{vanderPol_e1} and \ref{vanderPol_e13} represent the limit cycle (numerically found) and the transversal ellipses for $\varepsilon \, = \, 1/10$, $\varepsilon \, = \, 1$ and $\varepsilon \, = \, 13$ in the $(x,y)$-plane.

\begin{figure}[htb]
\includegraphics[width=0.5\textwidth]{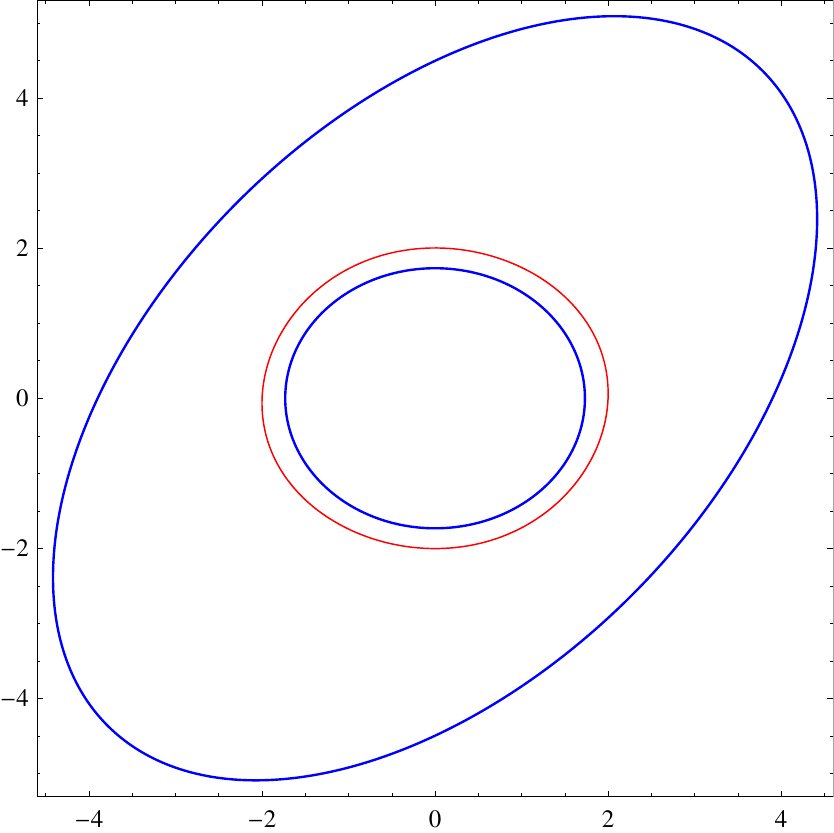}
\caption{Stable limit cycle in red and transversal ellipses in blue when $\varepsilon \, = \, 1/10$. The exterior ellipse corresponds to $a_{0e} \, = \, 39/10$.}
\label{vanderPol_e1s10}
\end{figure}

\begin{figure}[htb]
\includegraphics[width=0.5\textwidth]{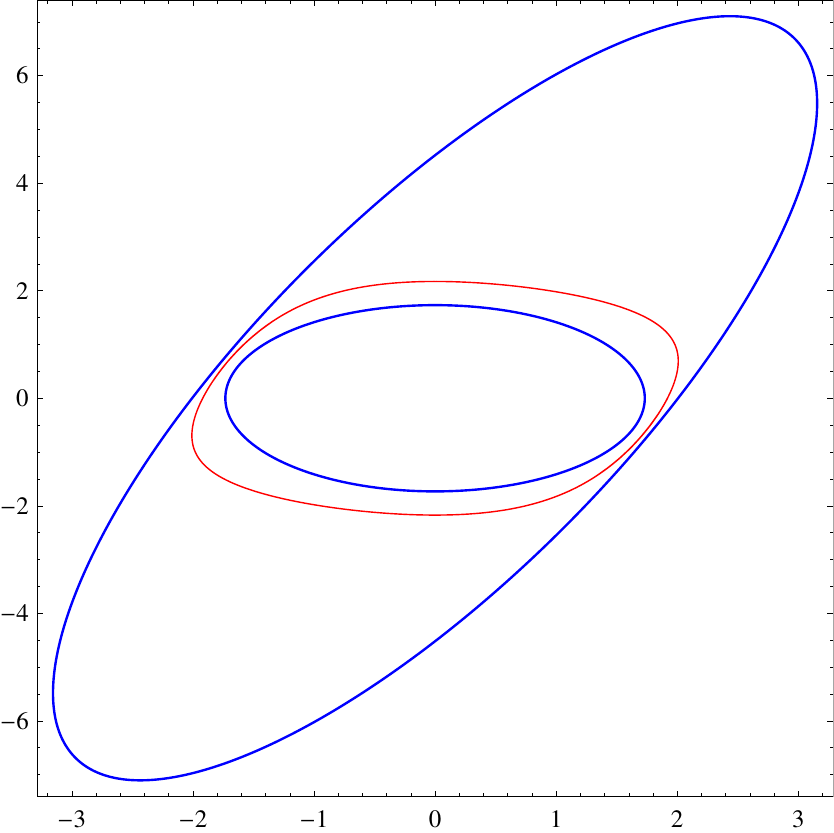}
\caption{Stable limit cycle in red and transversal ellipses in blue when $\varepsilon \, = \, 1$. The exterior ellipse corresponds to $a_{0e} \, = \, 2007/1000$.}
\label{vanderPol_e1}
\end{figure}

\begin{figure}[htb]
\includegraphics[width=0.5\textwidth]{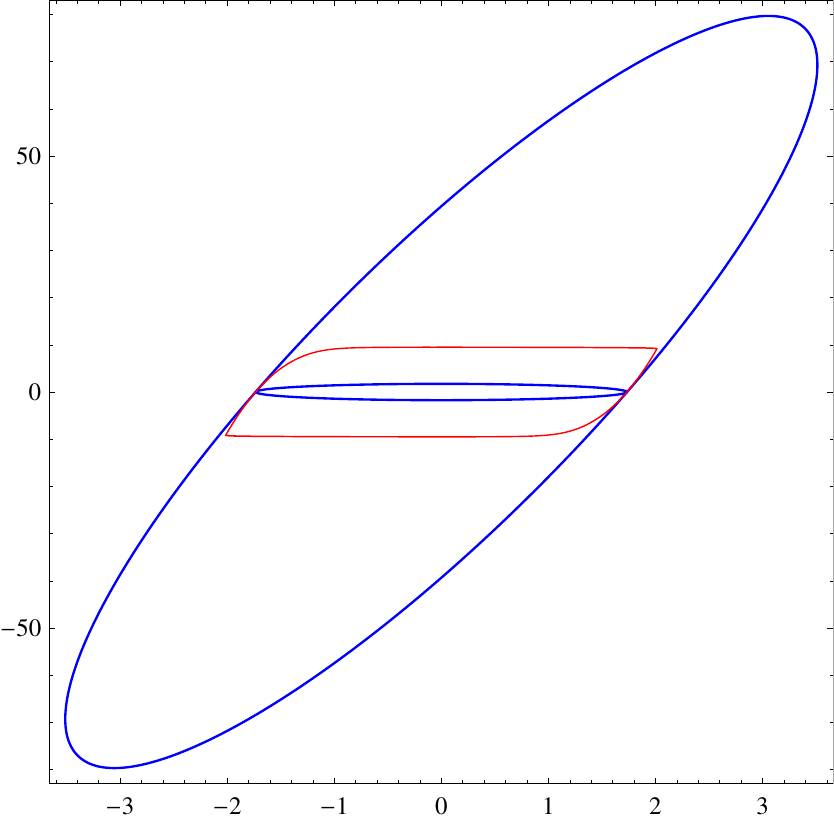}
\caption{Stable limit cycle in red and transversal ellipses in blue when $\varepsilon \, = \, 13$. The exterior ellipse corresponds to \[ a_{0e} \, = \, \displaystyle\frac{808+6960843825 \sqrt{3}}{6950604960}  \, = \, 1.734\ldots. \]
Note that $a_{0i} \, = \, \sqrt{3} \, = \, 1.732\ldots$.}
\label{vanderPol_e13}
\end{figure}

For system (\ref{eqex2}), there are several bands of transversal ellipses in the exterior of the region bounded by the limit cycle. We take the closest one to the origin as exterior ellipse of the Poincar\'e--Bendixson region. The size of the exterior ellipse grows to infinity when $\varepsilon \to 0$.
It does not bifurcate from the circle of radius $2$. \par For $\varepsilon \geq 8/5$, the cuts of the exterior ellipse and the limit cycle with the semiaxis $\{(x,0): x>0\}$ are very close. The exterior ellipse corresponds to
\[ a_{0e}(\varepsilon) \, = \, \left\{ \begin{array}{lll}
\displaystyle \left(\frac{6}{\varepsilon}\right)^{\frac{1}{3}} & \ & \mbox{if} \ 0<\varepsilon\leq 1/6,
\vspace{0.2cm} \\
\displaystyle 9/2 & \ & \mbox{if} \ 1/6<\varepsilon<8/5,
\vspace{0.2cm} \\
\displaystyle \sqrt{3} +\frac{\sqrt{3}}{4}\frac{1}{\varepsilon^2}-\frac{17}{32 \sqrt{3}}\frac{1}{\varepsilon^4}+\frac{1}{36\varepsilon^5} + \frac{1}{5\varepsilon^6}& \ & \mbox{if} \ 8/5\leq\varepsilon.
\end{array} \right. \]

We remark that to prove that the ellipse corresponding to the value
\begin{equation} \label{a0e} a_{0e}(\varepsilon) \, = \, \sqrt{3} +\frac{\sqrt{3}}{4}\frac{1}{\varepsilon^2}-\frac{17}{32 \sqrt{3}}\frac{1}{\varepsilon^4}+\frac{1}{36\varepsilon^5} + \frac{1}{5\varepsilon^6} \quad \mbox{if} \ 8/5\leq\varepsilon \end{equation}
is transversal, we need to apply Sturm algorithm to a polynomial of degree $71$. The latter expression of $a_{0e}(\varepsilon)$ comes from two roots of the resultant $R(a_0)$ of the numerator of $\varphi(m)$ and its derivative with $m$ as variable (see the fourth step in section \ref{sect2}). The values $a_0$ in the interval between these two roots of $R(a_0)$ give rise to a band of transversal exterior ellipses to the limit cycle. When $\varepsilon$ tends to infinity, these two roots coalesce. We study the asymptotic behavior of these two roots when $\varepsilon \to + \infty$ and we find the expression (\ref{a0e}) which belongs to the intermediate region between these two roots for any $\varepsilon \geq 8/5$. The most difficult case to show that $f=0$ is a transversal ellipse for the flow of system (\ref{eqex2}) is when $a_0=a_{0e}(\varepsilon)$ as given in (\ref{a0e}). Concretely, if we substitute this $a_{0e}(\varepsilon)$ in the expression of the polynomial $P_4$ defined in (\ref{eqfi2}), we get that the discriminant of this polynomial with respect to $m$ gives $H_{10}(\varepsilon) H_{20}(\varepsilon) \tilde{H}_{20}(\varepsilon) H_{32}(\varepsilon) H_{71}(\varepsilon)$ where $H_j$ is a polynomial in $\varepsilon$ of degree $j$, $j=10,20,32,71$, and $\tilde{H}_{20}$ is a polynomial in $\varepsilon$ of degree $20$, all of them with no real roots in the interval $\varepsilon \in [8/5,+\infty)$. We recall that in order to prove that a polynomial has no real roots on a certain interval we use Sturm sequences. To be more precise, consider the bivariate polynomial $p_4(m,\varepsilon) \, := \, P_4(m, a_{0e}(\varepsilon),\varepsilon)$. We want to show that $p_4(m,\varepsilon) \neq 0$ for all $m \in \mathbb{R}$ and for all $\varepsilon \geq 8/5$. The resultant with respect to $m$ of $p_4(m,\varepsilon)$ and its derivative with respect to $m$ is a polynomial in $\varepsilon$ which we denote by $\tilde{R}(\varepsilon)$. It is well-known, see for instance Chapter 1 of \cite{Walker}, that if a fixed $\varepsilon^{*}$ is such that $\tilde{R}(\varepsilon^{*})=0$ then either the polynomial $p_4(m,\varepsilon^{*})$ has a multiple root or $\varepsilon^{*}$ is a root of the coefficient of the monomial of highest degree in $m$ of $p_4(m,\varepsilon)$. As we have shown, in our case, the polynomial $R(\varepsilon) \neq 0$ for all $\varepsilon \geq 8/5$ and we can also check that for one fixed value $\tilde{\varepsilon} \geq 8/5$ we have that $p_4(m,\tilde{\varepsilon}) \neq 0$ for all $m \in \mathbb{R}$. Then, we conclude that $p_4(m,\varepsilon) \neq 0$ for all $m \in \mathbb{R}$ and $\varepsilon \geq 8/5$.

\smallskip

From the values of $a_{0i} \, = \, \sqrt{3}$ and $a_{0e}$ we get the following lemma.
\begin{lemma} \label{lem2}
For $\varepsilon \geq 8/5$, the cut of the limit cycle with the semiaxis $\{(x,0): x>0\}$ belongs to the interval
\[ \left(\sqrt{3}, \ \sqrt{3} +\frac{\sqrt{3}}{4}\frac{1}{\varepsilon^2}-\frac{17}{32 \sqrt{3}}\frac{1}{\varepsilon^4}+\frac{1}{36\varepsilon^5} + \frac{1}{5\varepsilon^6}\right), \] whose range tends to $0$ when $\varepsilon \to +\infty$.
\end{lemma}

We see that the transversal ellipses give a good approximation of the cut of the limit cycle of system (\ref{eqex2}) with the semiaxis $\{(x,0): x>0\}$ for big $\varepsilon$. However, we do not obtain a good approximation of the amplitude of the limit cycle, which is the greatest value of the $x$-coordinate in the limit cycle, for an arbitrary $\varepsilon>0$ because the transversal exterior ellipse tends to infinity when $\varepsilon \to 0^{+}$. The study of the amplitude of the limit cycle of system (\ref{eqex2}) depending on $\varepsilon$ is studied in \cite{Odani97}.

\subsection{Example 3: Rychkov system\label{sect33}}

We consider the system studied by Rychkov in 1975, see \cite{Rychkov75}.
\begin{equation} \label{eqex3} \dot{x} \, = \, y-\left( x^5-\mu x^3 + \delta x \right), \quad \dot{y} \, = \, -x, \end{equation} with $\delta, \mu \in \mathbb{R}$. This system is studied in \cite{Alsholm92, GiaNeu98, Odani96, Rychkov75}. The following features of system (\ref{eqex3}) can be found in the aforementioned references. The origin is the only finite singular point and it is a focus. Rychkov \cite{Rychkov75} proved that it has at most two limit cycles and that for $\delta<0$ there exists a unique limit cycle, which is stable. The line $\delta=0$ is a curve of occurrence of Hopf bifurcation. When $\delta>0$ there is a curve of bifurcation values $\mu \, = \, \mu^{*}(\delta)$ of a sky-blue bifurcation. Odani \cite{Odani96} proved that if $\delta>0$ and $\mu > \sqrt{5\delta}$, then the system has two limit cycles. Figure \ref{Rychov_bif} represents the bifurcation diagram of the Rychkov system (\ref{eqex3}) in the $(\delta,\mu)$-plane. Recall that a sky-blue bifurcation occurs when a stable limit cycle and an unstable limit cycle coalesce and become a semi-stable limit cycle. In a neighborhood of the values of the parameters where such a sky-blue bifurcation is exhibited and in a neighborhood of the semi-stable limit cycle, there are either no limit cycles, a semi-stable limit cycle or a stable and unstable couple of limit cycles. A sky-blue bifurcation corresponds to an elementary catastrophe of fold type.

\begin{figure}[htb]
\includegraphics[width=0.5\textwidth]{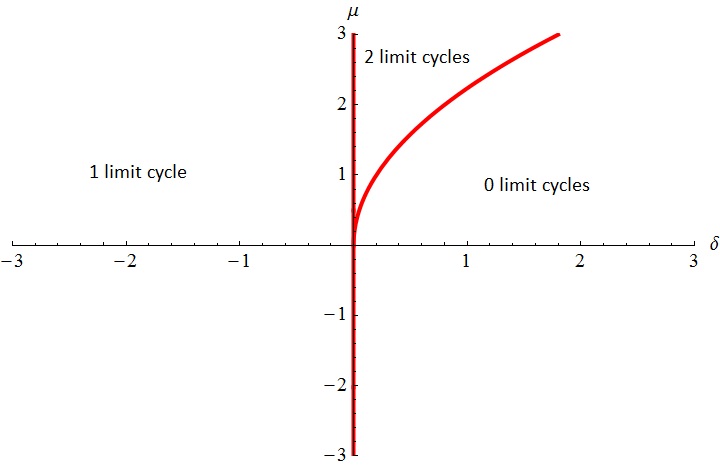}
\caption{Bifurcation diagram of the Rychkov system (\ref{eqex3}).}
\label{Rychov_bif}
\end{figure}

We fix $\mu=1$, that is we take
\begin{equation} \label{eqex31} \dot{x} \, = \, y-\left( x^5- x^3 + \delta x \right), \quad \dot{y} \, = \, -x, \end{equation} where $\delta \in \mathbb{R}$. We note that the system is a semi-complete family of rotated vector fields with respect to $\delta$. Numerically, the sky-blue bifurcation value is $\delta^{*} \, \approx \, 0.225$ and the bound given by Odani gives $\delta^{*}>0.2$. We choose different values of $\delta$ and we use the method described in section \ref{sect2} in order to find, when possible, Poincar\'e--Bendixson regions which locate the limit cycle $(\delta<0)$ or the two limit cycles $(0<\delta<\delta^{*})$.

\smallskip

In this case we take as initial condition $(a_0,0)$ and a conic $f=0$ of the form $f(x,y)\, =\, 1 + s_3 x^2 + s_4 xy + s_5 y^2$ because the system is symmetric with respect to the origin. After the first and second steps described in section \ref{sect2}, we get
\[ \begin{array}{lll}
f(x,y) & = & \displaystyle 1 \, -\,  \frac{x^2}{a_0^2} \, +\, \frac{2(-a_0^2+a_0^4+\delta)xy}{a_0^2}
\vspace{0.2cm} \\
& & \displaystyle
 + \, \frac{(-1+a_0^4-4a_0^6+3a_0^8+2a_0^4\delta-\delta^2)y^2}{a_0^2}.
\end{array} \]

When $\delta \, = \, -1$, we find a Poincar\'e--Bendixson region bounded by two transversal ellipses. Figure \ref{casoRychkov_dm1} represents the stable limit cycle (numerically found) in red and the two transversal ellipses in blue, in the $(x,y)$-plane.

\begin{figure}[htb]
\includegraphics[width=0.5\textwidth]{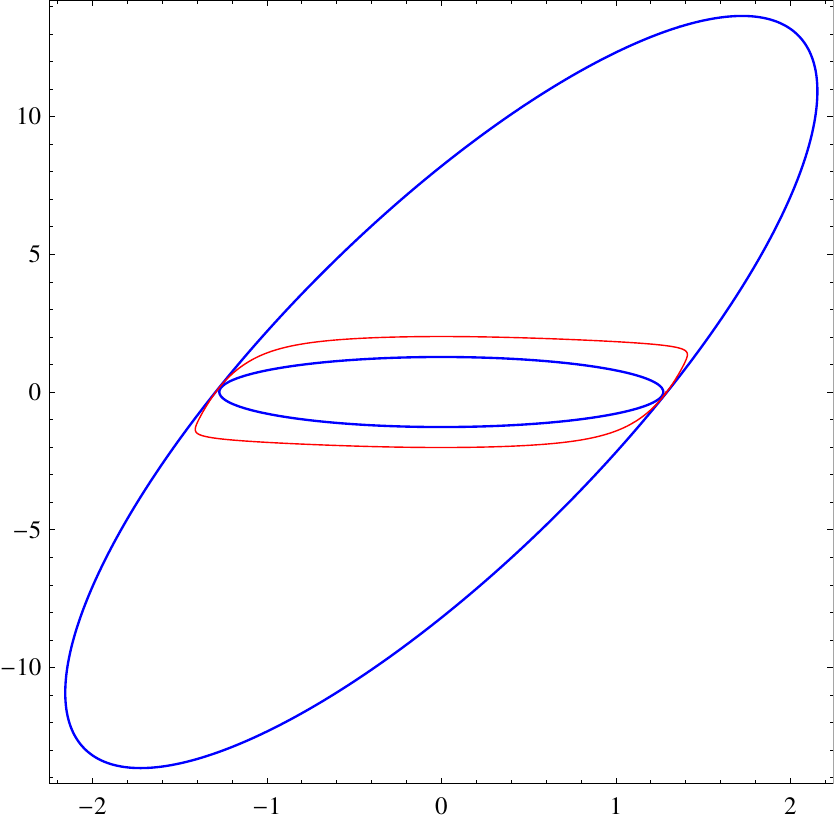}
\caption{Stable limit cycle in red and transversal ellipses in blue when $\delta \, = \, -1$. \[ a_{0i} \, = \, \frac{159}{125} \, = \, 1.272, \quad a_{0e} \, = \, \frac{646747}{500000} \, \approx \, 1.293.\]}
\label{casoRychkov_dm1}
\end{figure}

When $\delta \, = \, 1/10$, $\delta \, = \, 15/100$ or $\delta\, = \, 18/100$, we find two Poincar\'e--Bendixson regions bounded by three transversal ellipses, which separate the two limit cycles. Figures \ref{casoRychkov_d1s10}, \ref{casoRychkov_d15s100} and \ref{casoRychkov_d18s100} represent each of these cases, respectively. The stable limit cycle is represented in red and the unstable limit cycle in green. These limit cycles are numerically found. In each figure we represent the three transversal ellipses in blue: $a_{0i}$, $a_{0m}$ and $a_{0e}$ are the values corresponding to the interior, in the middle and exterior, respectively, ellipses bounding the Poincar\'e--Bendixson regions. The following table contains the values of $a_{0i}$, $a_{0m}$ and $a_{0e}$ for each value of $\delta \, = \, 1/10$, $\delta \, = \, 15/100$ and $\delta\, = \, 18/100$.

\[ \begin{array}{c|c|c|c}
\delta & a_{0i} & a_{0m} & a_{0e}  \vspace{0.2cm} \\ \displaystyle \frac{1}{10} & \displaystyle \frac{67}{200} \, = \, 0.335 & \displaystyle  \frac{937}{1000} \, = \, 0.937 & \displaystyle \frac{1081}{1000} \, = \, 1.081 \vspace{0.2cm} \\ \displaystyle \frac{15}{100} & \displaystyle \frac{2143}{5000} \, \approx \, 0.429 & \displaystyle \frac{22097}{25000} \, \approx \, 0.884 & \displaystyle \frac{10687}{10000} \, \approx \, 1.069 \vspace{0.2cm} \\ \displaystyle \frac{18}{100} & \displaystyle \frac{1213}{2500} \, \approx \, 0.485 &
\displaystyle \frac{80337}{100000} \, \approx \, 0.803 & \displaystyle \frac{26529}{25000} \, \approx \, 1.061
\end{array} \]

\begin{figure}[htb]
\includegraphics[width=0.5\textwidth]{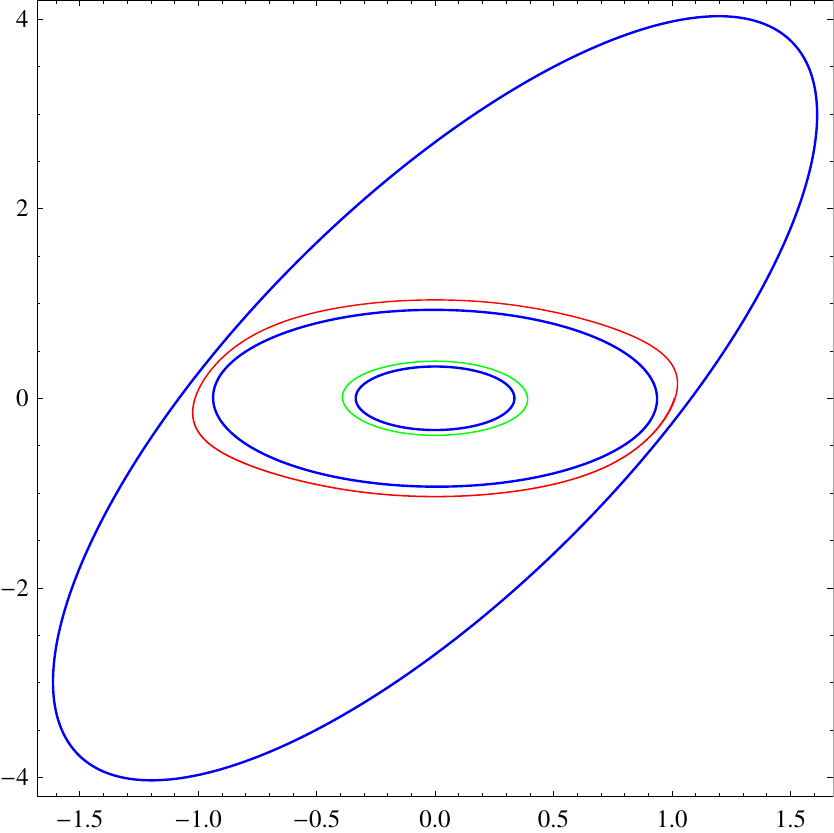}
\caption{Rychkov system {\rm (\ref{eqex31})} when $\delta \, = \, 1/10$. Transversal ellipses in blue. Stable limit cycle in red and unstable limit cycle in green.}
\label{casoRychkov_d1s10}
\end{figure}

\begin{figure}[htb]
\includegraphics[width=0.5\textwidth]{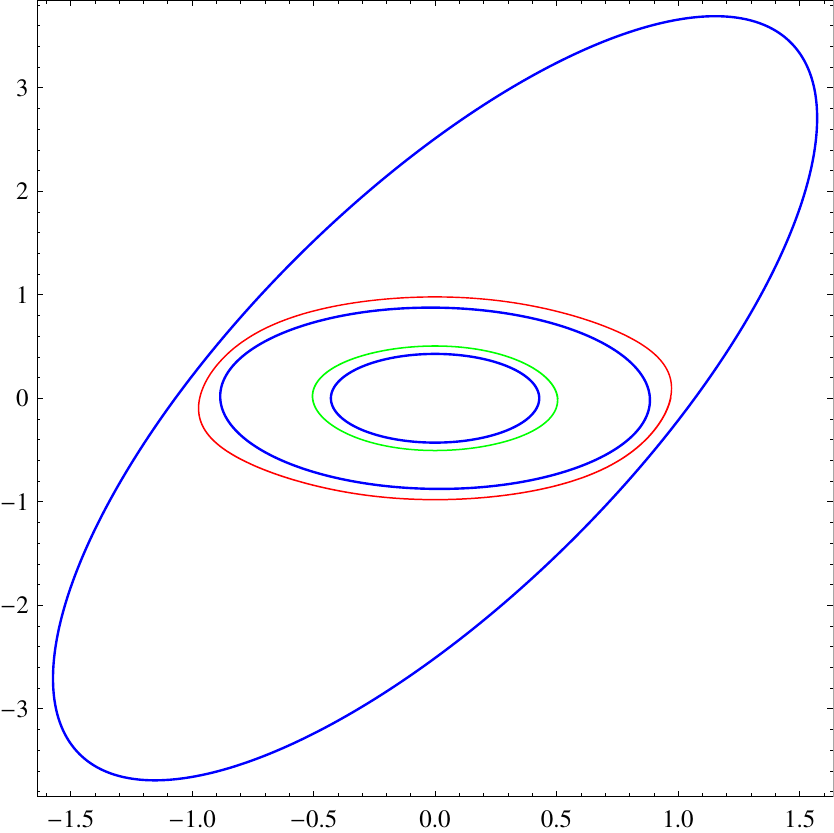}
\caption{Rychkov system {\rm (\ref{eqex31})} when $\delta \, = \, 15/100$. Transversal ellipses in blue. Stable limit cycle in red and unstable limit cycle in green.}
\label{casoRychkov_d15s100}
\end{figure}

\begin{figure}[htb]
\includegraphics[width=0.5\textwidth]{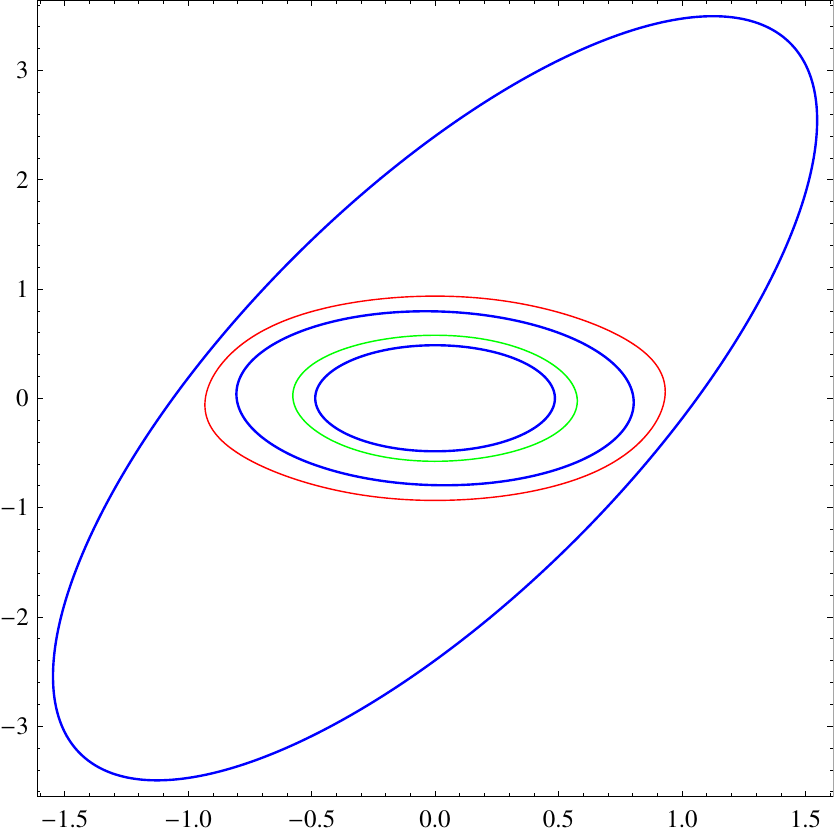}
\caption{Rychkov system {\rm (\ref{eqex31})} when $\delta \, = \, 18/100$. Transversal ellipses in blue. Stable limit cycle in red and unstable limit cycle in green.}
\label{casoRychkov_d18s100}
\end{figure}

Using the same method, we can find the two Poincar\'e--Bendixson regions bounded by three transversal ellipses, which separate the two limit cycles, for $\delta=0.1815$. We can cover all the interval $\delta \in (0, 0.1815)$ because we have found the three transversal ellipses for $\delta=0.1815$. Since system (\ref{eqex31}) is a semi-complete family of rotated vector fields with respect to $\delta$ we deduce the existence of these two limit cycles for values of $\delta$ lower than $0.1815$. Recall that the bifurcation value is $\delta^{*} \, \approx \, 0.225$ and the bound given by Odani gives $\delta^{*}>0.2$. Taking as initial condition $(a_0,b_0)$, with $b_0=-1/26$ and $a_0 \neq 0$, we get three transversal ellipses, which separate the two limit cycles, for $\delta \in (0,0.19991)$. It turns out that two of the transversal ellipses annihilate one each other when $\delta=0.19991$. We could say that {\em there is a sky-blue bifurcation of transversal ellipses} before the sky-blue bifurcation of limit cycles. The ellipses that coalesce when $\delta=0.19991$ belong to the band inside the region bounded by the two limit cycles. We remark that we have not been able to give a close bound for the sky-blue bifurcation of limit cycles. This may be due to the fact that when the limit cycles are very close, the intermediate region is too narrow to contain a transversal conic.

\subsection{Example 4: A quasi-homogeneous system\label{sect34}}

We consider the system studied in \cite{GasGia06}
\begin{equation} \label{eqex4} \dot{x} \, = \, y, \quad \dot{y} \, = \, -x^3+\delta x^2 y+y^3, \end{equation} with $\delta \in \mathbb{R}$. In \cite{CGM97} it is proved that the origin is the only finite singular point and it is a focus which is stable when $\delta<0$ and unstable when $\delta \geq 0$. Indeed, in \cite{CGM97} it is also proved that for $\delta <0$ and close to zero, there exists at least one limit cycle born by Hopf bifurcation from the bifurcation value $\delta=0$. In \cite{GasGia06} it is proved that for $\delta \geq 0$ there are no limit cycles using Bendixson criterion; for $\delta <-2.679$ there are no limit cycles and for $-\sqrt[3]{27/2} < \delta <0$, there is at most one limit cycle. Note that $-\sqrt[3]{27/2} \approx -2.381$. Numerically, the bifurcation value is $\delta^{*} \, \approx \, -2.198$. We note that system (\ref{eqex4}) is a semi-complete family of rotated vector fields with respect to $\delta$. It is known that for $\delta \approx 0$, $\delta <0$ there exists a limit cycle born by Hopf bifurcation when $\delta =0$ but an open problem was to prove its existence for $\delta \in (\delta^{*},0)$. We partially solve this problem for a real interval contained in $(\delta^{*},0)$.

\begin{lemma} \label{lem3}
For  $\displaystyle 0>\delta > -3\sqrt[3]{-9+6\sqrt{3}}/\sqrt[3]{4} \, \approx \, -2.1103, $ we can find a Poincar\'e--Bendixson region for system {\rm (\ref{eqex4})} bounded by the origin (which is stable) and a transversal hyperbola together with part of the equator of the Poincar\'e disk.
\end{lemma}

This lemma is proved by the arguments contained in this subsection. We use the method described in section \ref{sect2} and in this case we take the initial condition $(0,b_0)$ and a conic $f=0$ of the form $f(x,y) \, = \, 1 + s_3 x^2 + s_4 xy + s_5 y^2$. The transversal hyperbola $f=0$ is of the form \begin{equation} \label{eqfex4} f(x,y) \, = \, 1 + b_0^2 x^2+ 2xy-\frac{y^2}{b_0^2}, \end{equation} with $b_0>0$ the lowest root of the polynomial $1-4b_0^3+b_0^6-b_0^2 \delta$. \par We have tried to find transversal conics $f=0$ of the form $f(x,y) \, = \, 1 + s_3 x^2 + s_4 xy + s_5 y^2 $ and of the form $ f(x,y) \, = \, 1 +  s_1 x + s_2 y+ s_3 x^2 + s_4 xy + s_5 y^2  $ with initial condition $(a_0,b_0)$ in order to find transversal ellipses but we have not succeeded.

\smallskip

Figure \ref{p4pacificdm2} represents the phase portrait of system (\ref{eqex4}) for $\delta = -2$ and it shows the limit cycle and the equator of the Poincar\'e disk, done with program P4. Figure \ref{casoPacific_dm2} represents the finite part of the Poincar\'e--Bendixson region in the $(x,y)$-plane, also for $\delta = -2$. The unstable limit cycle is represented in green and the transversal hyperbola in blue.

\begin{figure}[htb]
\includegraphics[width=0.5\textwidth]{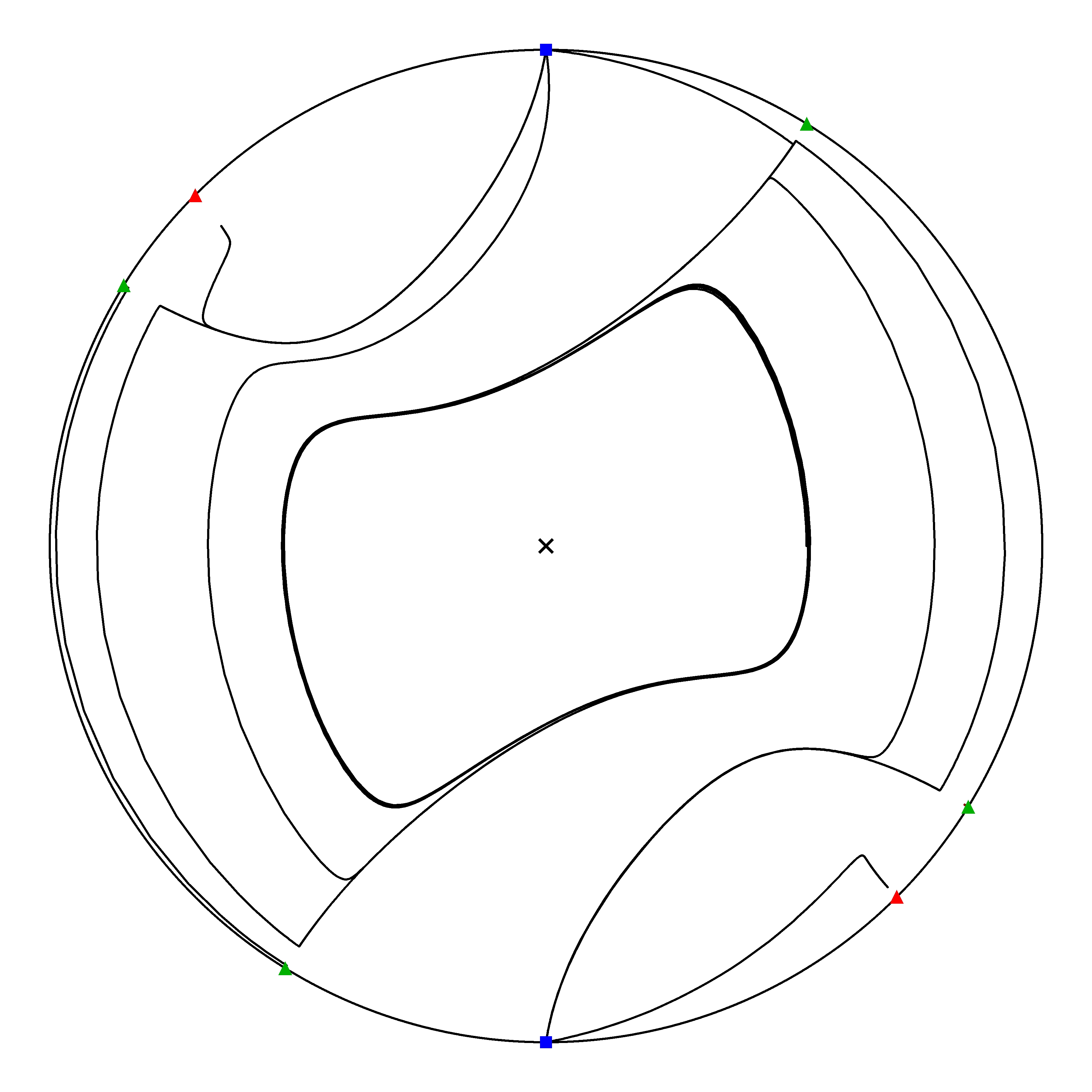}
\caption{Phase-portrait in the Poincar\'e disk of system {\rm (\ref{eqex4})} when $\delta \, = \, -2$.}
\label{p4pacificdm2}
\end{figure}

\begin{figure}[htb]
\includegraphics[width=0.5\textwidth]{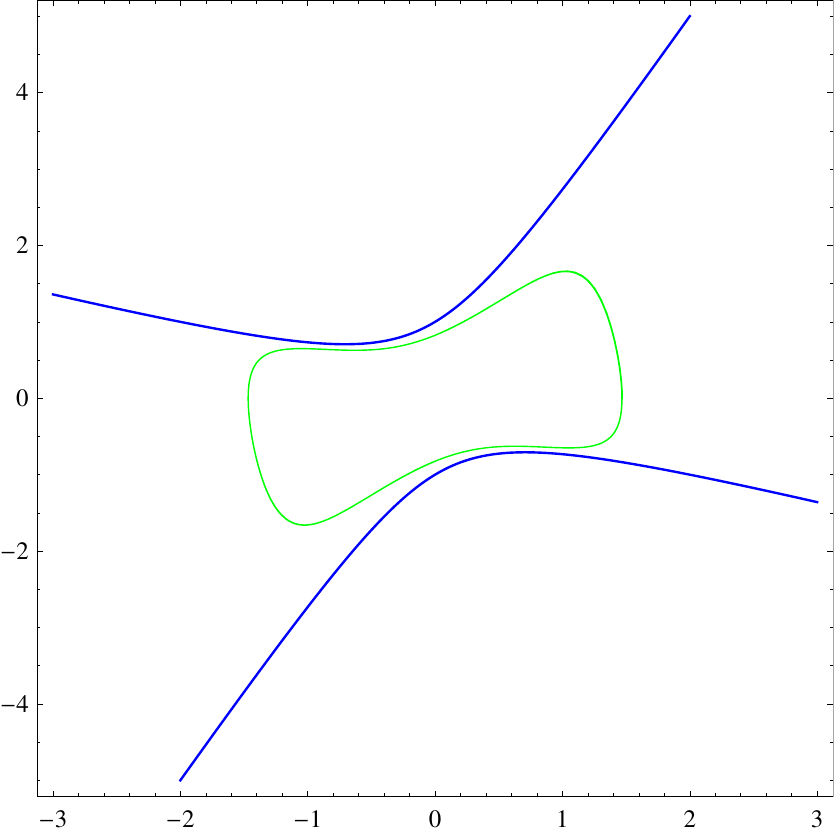}
\caption{Phase-portrait in the $(x,y)$-plane of system {\rm (\ref{eqex4})} when $\delta \, = \, -2$. Unstable limit cycle in green and transversal hyperbola in blue, $b_0=1$. }
\label{casoPacific_dm2}
\end{figure}

We need to study the type of singular points ``at infinity'', that is, in the equator of the Poincar\'e disk because these points may belong to the boundary of the Poincar\'e--Bendixson region and we need to ensure that they are not sinks or sources which will not fulfill the hypothesis of Theorem \ref{thpb}.  Due to the symmetries of the Poincar\'e compactification and the fact that the real plane is embedded twice in the sphere $\mathbb{S}^2$, if a point in the equator of the Poincar\'e-disk of the form $(X_0,Y_0,0)$ is a singular point for an embedded differential system (\ref{eq1}), then $(-X_0,-Y_0,0)$ is also a singular point and the behavior of the orbits of both points is also symmetric with respect to the involution $(X,Y,0) \to (-X,-Y,0)$. Thus, we only need to study the behavior of the singular points on $\mathbb{S}^2$ of the form $(X_0,Y_0,0)$ with $Y_0 \geq 0$. It is easy to see that the points $(\pm 1,0,0) \in \mathbb{S}^2$ are regular points for system (\ref{eqex4}) once embedded in the sphere. Thus, in order to embed system (\ref{eqex4}) in the sphere, we consider a local chart $(u,v)$ in $\mathbb{S}^2$ which corresponds to the open set $U \, = \, \left\{  (X,Y,Z) \in \mathbb{S}^2 \, : \, Y>0\right\}$ and the local map \[ \begin{array}{cccl} \phi: & U & \longrightarrow & \mathbb{R}^2 \\ & (X,Y,Z) & \mapsto & \displaystyle \left( \frac{X}{Y}, \, \frac{Z}{Y} \right) \, = \, (u,v). \end{array} \] This local chart corresponds to the notation $(U_2,\phi_2)$ in Chapter 5 of the book \cite{DuLlAr}. Thus, the change $x=u/v$, $y=1/v$ in system (\ref{eqex4}) and a rescalling of time, leads to the system
\begin{equation} \label{eqex4i}
\dot{u} \, = \, -u-\delta u^3+u^4+v^2, \quad \dot{v} \, = \, -(1+\delta u^2-u^3) v.
\end{equation}
The singular points ``at infinity'' for system (\ref{eqex4}) are the singular points of system (\ref{eqex4i}) on the straight line $v=0$. Note that this straight line is the intersection of the equator of the Poincar\'e disk with the open set $U$ of the local chart $(U,\phi)$. The $(u,v)$ coordinates of the singular points of system (\ref{eqex4i}) on the line $v=0$ are $(0,0)$ and $(u_i,0)$ with $u_i$ a root of the polynomial $u_i^3-\delta u_i^2-1$, $i=1,2,3$. Easy computations show that the point $(0,0)$ is a hyperbolic stable node, $\forall \delta$. The polynomial $u^3-\delta u^2-1$ always has a real root $u_1>0$, $\forall \delta$, which gives rise to a semi-hyperbolic saddle. For the classification of semi-hyperbolic singular points of a planar real differential system, see for instance Chapter 2 of \cite{DuLlAr}. When $0\, >\, \delta\, >\, -3/\sqrt[3]{4}\, \approx\, -1.8899$, we have no other singular points at the equator of the Poincar\'e disk. When $\delta\, =\, -3/\sqrt[3]{4}$, we have another singular point at $u=-\sqrt[3]{2}$, which gives a non-elementary singular point with parabolic sectors where the orbits leave from it. When $\delta<-3/\sqrt[3]{4}$, we have two other singular points $0>u_2>u_3$ which give $u_2$ a semi-hyperbolic unstable node and $u_3$ a semi-hyperbolic saddle.

\smallskip

We consider the polynomial $f(x,y)$ which gives rise to the transversal hyperbola $f=0$ given in (\ref{eqfex4}) and we consider it in the local chart $(U,\phi)$. We define $\tilde{f}(u,v) \, := \, v^2 \, f (u/v,\, 1/v)$, and we get that the transversal hyperbola in this chart is $\tilde{f}(u,v)=0$ with
\[ \tilde{f}(u,v) \, = \, - \, \frac{1}{b_0^2} \, + \, 2u + b_0^2u^2+v^2. \] Note that $f(0,0)=1$, so the points in the equator of the Poincar\'e disk which belong to the same region need to satisfy $\tilde{f}(u,v)>0$. We note that $\tilde{f}(0,0) = -1/b_0^2$, so the hyperbolic stable node is never in the Poincar\'e--Bendixson region. It is easy to show that $\tilde{f}(u_1,0)>0$, so the Poincar\'e--Ben\-dix\-son region always contains the semi-hyperbolic saddle $(u_1,0)$. A saddle is neither a sink or a source and, thus, this point may belong to the boundary of a Poincar\'e--Bendixson region. \par We denote $\delta_i \, = \, \displaystyle -3\sqrt[3]{-9+6\sqrt{3}}/\sqrt[3]{4} \, \approx \, -2.1103$. When $\delta > \delta_i$, we have that the semi-hyperbolic unstable node $(u_2,0)$ is such that $\tilde{f}(u_2,0)<0$ (outside the Poincar\'e--Bendixson region). When $\delta \leq \delta_i$, we no longer have a transversal hyperbola, so we have not the Poincar\'e--Bendixson region constructed with such a conic. We have shown the existence of a Poincar\'e--Bendixson region for $\delta \, = \, \delta_i$ and since system (\ref{eqex4}) is a semi-complete family of rotated vector fields with respect to $\delta$, we deduce the existence of a limit cycle for $\delta \in (\delta_i,0)$.

\subsection{Example 5: A quintic system with $1$ limit cycle\label{sect35}}

We consider the system studied in \cite{GGG14}:
\[ \dot{x} \, = \, y, \quad \dot{y} \, = \, -x+(a-x^2)(y+y^3), \]
with $a \in \mathbb{R}$. When $a\leq 0$, there are no limit cycles. When $a>0$, we take $a=b^2$. Thus, we consider the system
\begin{equation} \label{eqex5}  \dot{x} \, = \, y, \quad \dot{y} \, = \, -x+(b^2-x^2)(y+y^3), \end{equation} with $b>0$. In \cite{GGG14} it is proved that the origin is the only finite singular point and it is an unstable focus. The system is a semi-complete family of rotated vector fields with respect to $b^2$. For $b \in (0,b^*)$ there exists a unique limit cycle, which is hyperbolic and stable. For $b \geq b^*$ there are no limit cycles. The bifurcation value $b^*$ is such that $0.79<b^*<0.817$. Numerically, the bifurcation value is $b^* \, \approx \, 0.80629$.

\smallskip

We consider certain values of $b$ and, applying the method described in section \ref{sect2}, we find a Poincar\'e--Bendixson region bounded by a transversal ellipse in the interior of the region bounded by the limit cycle and a transversal hyperbola together with part of the Poincar\'e disk in its exterior. In this case, we consider an initial condition $(0,b_0)$ and a conic $f=0$ of the form $f(x,y)\, =\, 1 + s_3 x^2+s_4 xy + s_5 y^2$. After applying steps number one and number two of section \ref{sect2}, we get that
\[ f(x,y) \, = \, 1 \, + \, \frac{\left(b^4 b_0^4-b^4-1\right)x^2}{b_0^2} \, + \, \frac{2 b^2 \left(b_0^2+1\right) x y}{b_0^2} \, - \, \frac{y^2}{b_0^2}. \]

Figures \ref{p4Johannab05} and \ref{casoJohanna_b05} correspond to $b=1/2$ and Figures \ref{p4Johanna_b065349} and \ref{casoJohanna_b065349} correspond to $b \, =\, 65349/100000$. In each case, we first show the phase portrait of system (\ref{eqex5}) in the Poincar\'e disk, figures done with the program P4, and the second figure of each case provides the finite part of the transversal hyperbola together with the transversal ellipse (both in blue) and the limit cycle (numerically found) in red represented in the $(x,y)$-plane.

\begin{figure}[htb]
\includegraphics[width=0.5\textwidth]{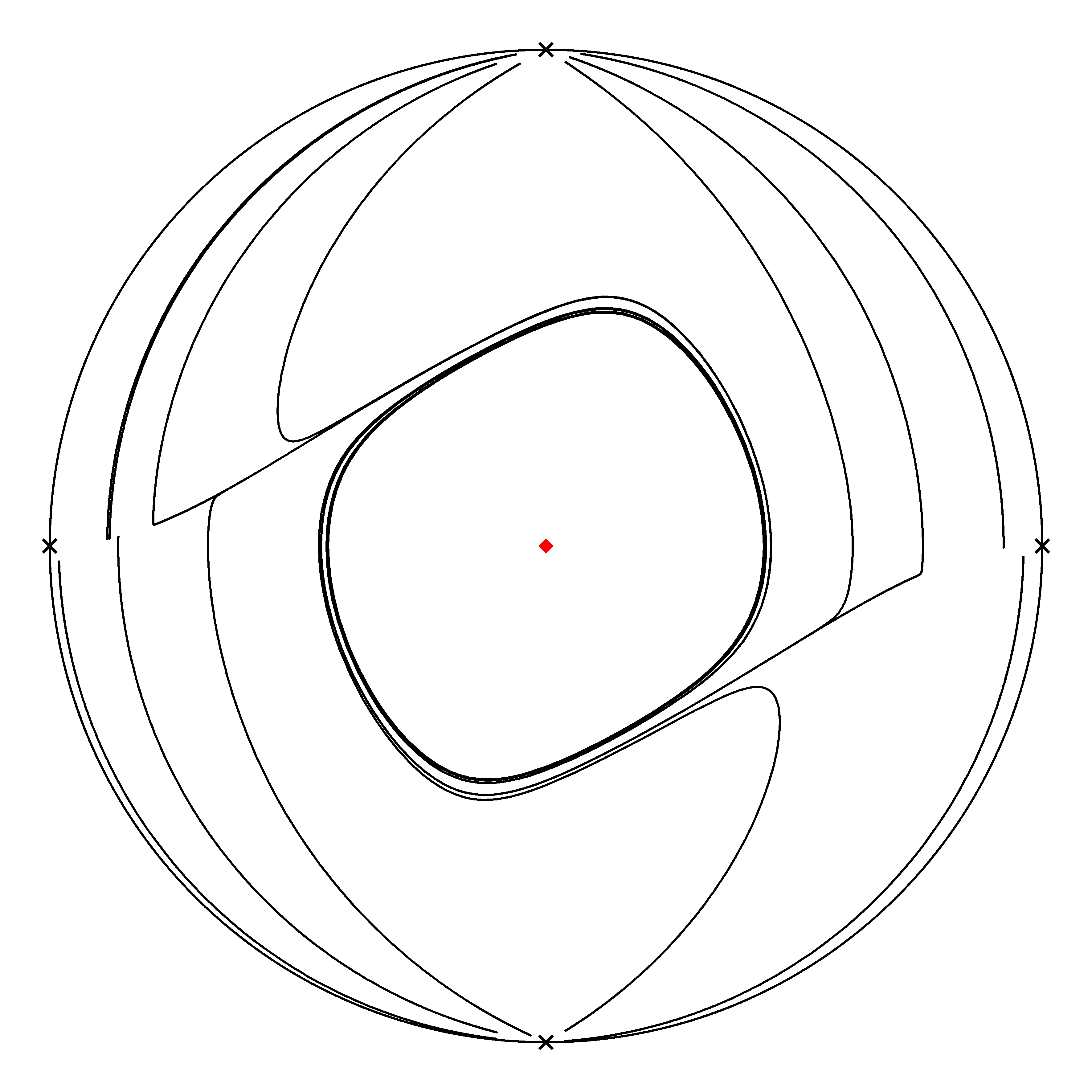}
\caption{Phase-portrait in the Poincar\'e disk of system {\rm (\ref{eqex5})} when $b\, = \, 1/2$.}
\label{p4Johannab05}
\end{figure}

\begin{figure}[htb]
\includegraphics[width=0.5\textwidth]{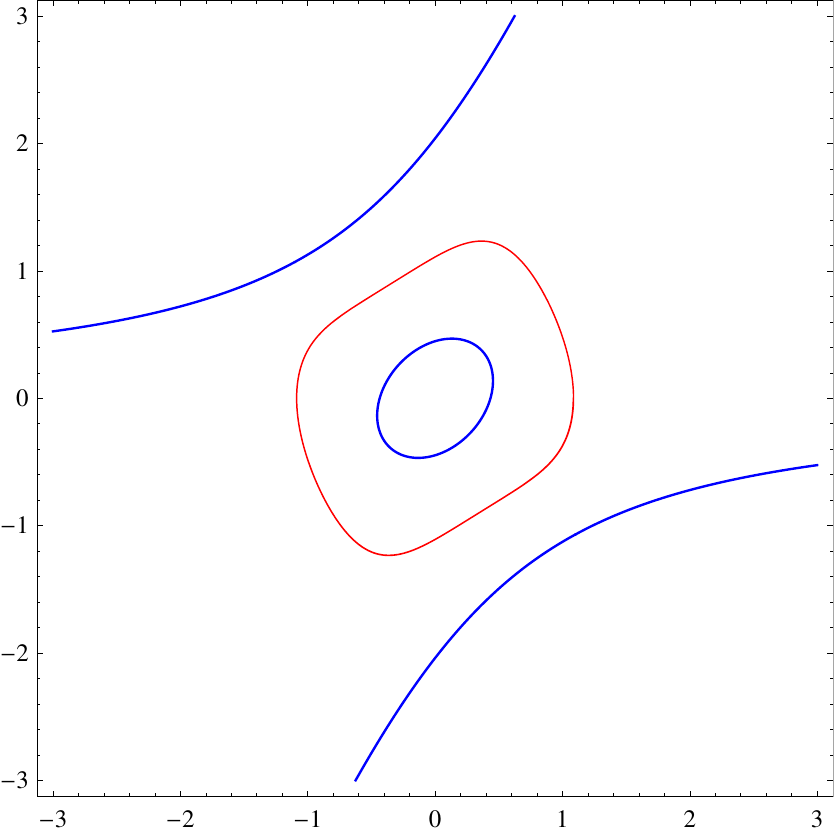}
\caption{System {\rm (\ref{eqex5})} with $b \, = \, 1/2$. Stable limit cycle in red. The transversal ellipse corresponds to $b_0 \, = \, \frac{2239}{5000} \, \approx \, 0.448$ and the transversal hyperbola to $b_0 \, = \, \frac{51}{25} \, = \, 2.04.$}
\label{casoJohanna_b05}
\end{figure}

\begin{figure}[htb]
\includegraphics[width=0.5\textwidth]{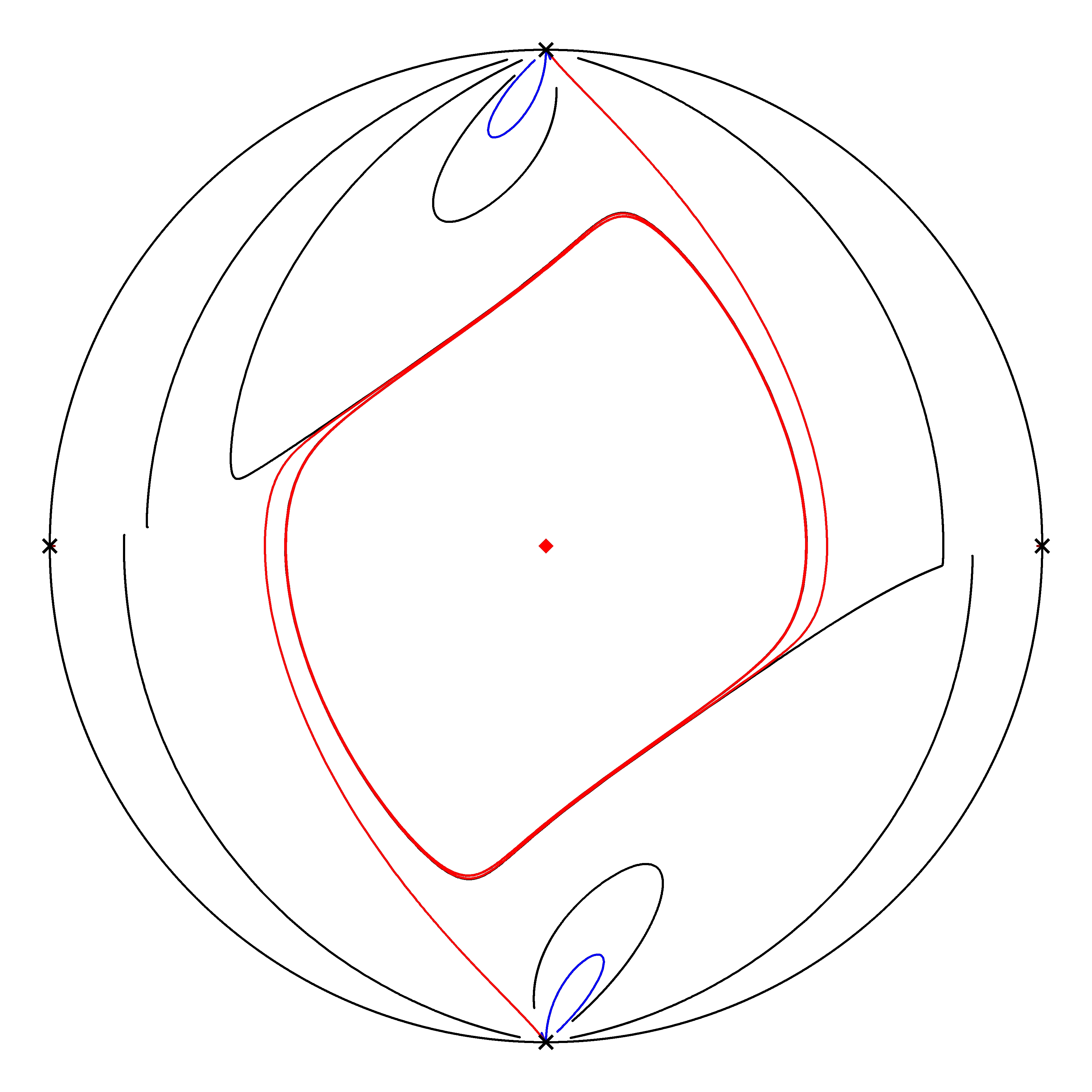}
\caption{Phase-portrait in the Poincar\'e disk of system {\rm (\ref{eqex5})} when $b\, = \, 0.65349$.}
\label{p4Johanna_b065349}
\end{figure}

\begin{figure}[htb]
\includegraphics[width=0.5\textwidth]{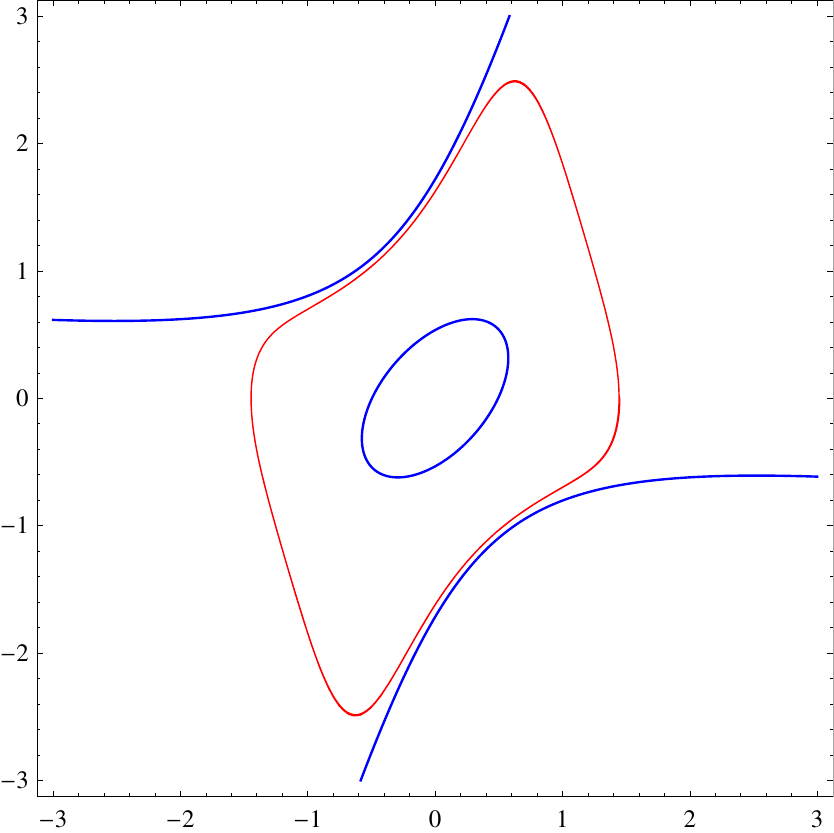}
\caption{System {\rm (\ref{eqex5})} with $b \, = \, 0.65349$. Stable limit cycle in red. The transversal ellipse corresponds to $b_0 \, = \, \frac{107}{200} \, = \, 0.535$ and the transversal hyperbola to $b_0 \, = \, \frac{859}{500} \, = \, 1.718.$}
\label{casoJohanna_b065349}
\end{figure}

There is a value $b>\frac{65349}{100000}$ and close to it for which the external transversal hyperbola does no longer exist. For $b= \frac{65349}{100000}$ we have a Poincar\'e--Bendixson region formed by an inner transversal ellipse and an outer transversal hyperbola together with part of the Poincar\'e disk. Since system (\ref{eqex5}) is a semi-complete family of rotated vector fields with respect to $b^2$, if we have a Poincar\'e--Bendixson region for $b=0.65349$, then we can follow the limit cycle with $b$ and prove the existence of a limit cycle for $0<b\leq \frac{65349}{100000}$. To prove the existence of this Poincar\'e-Bendixson region for $b=0.65349$, and analogously to the proof in the previous subsection for system (\ref{eqex4}), we need to study the singular points ``at infinity'' and the behavior of the orbits in a neighborhood of them. This study is already performed in \cite{GGG14}, we just give a summary of the facts that we need for our interests. The singular points at the equator of the Poincar\'e disk are $(0,\pm 1,0) \in \mathbb{S}^2$ and $(\pm 1,0,0) \in \mathbb{S}^2$. We have that the points $(0,\pm 1,0) $ are outside the Poincar\'e--Bendixson region that we consider for any $b>0$. The singular point $(1,0,0)$ is non-elementary and it can be shown that it is of saddle type (analogously for $(-1,0,0)$). These points are contained in the boundary of the considered Poincar\'e--Bendixson region for any $b>0$, but since they are not sinks nor sources, they allow to take the segment in the equator of the Poincar\'e disk as boundary of the Poincar\'e--Bendixson region. The study of the behavior of the orbits in a neighborhood of a non-elementary singularity can be done thorough a desingularization process, see Chapter 3 of \cite{DuLlAr}. \par Despite the bound that we encounter for the bifurcation value $b^{*}$ is worse than the one provided in \cite{GGG14}, our approach is far easier.

\subsection{Example 6: A quintic system with $2$ limit cycles\label{sect36}}

We consider the following differential system studied in \cite{GasGia10}.
\begin{equation} \label{eqex6} \begin{array}{lll}
\dot{x} & = & \displaystyle -y+4x -\frac{49}{10} x^3-\frac{26}{5} xy^2+\frac{1}{5} x^2y^2 + x^5+2x^3y^2+xy^4, \vspace{0.2cm}  \\
\dot{y} & = & \displaystyle x+4y -\frac{23}{5} x^2y-5y^3-\frac{1}{5} xy^3-\frac{2}{15} y^4+x^4y+2x^2y^3+y^5.
\end{array} \end{equation}
In \cite{GasGia10} it is proved that the origin is the only finite singular point and it is an unstable focus and that system (\ref{eqex6}) has exactly $2$ limit cycles.

In this example we find two Poincar\'e--Bendixson regions which allow to show the existence of at least one limit cycle in each of them. We apply the method described in section \ref{sect2} and in this case, we consider an initial condition $(a_0,0)$ and a conic $f=0$ of the form $ f(x,y)\, =\, 1 + s_3 x^2+s_4 xy + s_5 y^2.$ We obtain the following expression for the polynomial $f(x,y)$:
\[ \begin{array}{lll}
f(x,y) & = & \displaystyle 1 \, - \, \frac{x^2}{a_0^2} \, + \, \frac{\left(10 a_0^4-49 a_0^2+40\right) x y}{5a_0^2}
\vspace{0.2cm} \\ & & \displaystyle \frac{\left(200 a_0^8-1010 a_0^6+147 a_0^4+3800 a_0^2-3300\right) y^2}{100 a_0^2}. \end{array} \]

The outer boundary is formed by a transversal ellipse. The boundary between the two limit cycles is formed by two bands of transversal ellipses. The inner boundary is formed by a band of transversal ellipses, starting at the origin. There are several bands of transversal ellipses and, in Figure \ref{casometodoHA2}, we have chosen more ellipses than the ones needed to give two Poincar\'e--Bendixson regions because we want to restrict, as close as possible, the location of the limit cycles. Figure \ref{casometodoHA2} shows the transversal boundaries of the Poincar\'e--Bendixson regions (in blue), the stable limit cycle in red and the unstable limit cycle in green (both found numerically), represented in the $(x,y)$-plane. The transversal ellipses correspond to the values
\[ a_0 \, = \, \frac{1015}{1000}, \quad a_0 \, = \, \frac{10189}{10000}, \quad a_0 \, = \, \frac{196531}{100000} \quad \mbox{and} \quad a_0 \, = \, \frac{196665}{100000}. \]

\begin{figure}[htb]
\includegraphics[width=0.5\textwidth]{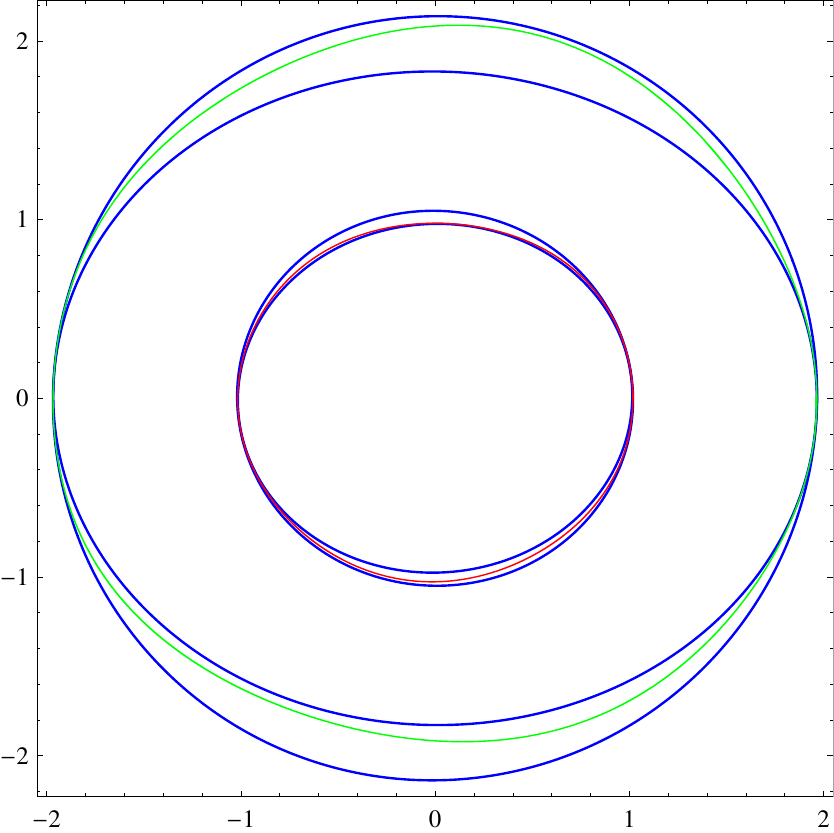}
\caption{System {\rm (\ref{eqex6})}. Stable limit cycle in red, unstable limit cycle in green. Boundaries of the Poincar\'e--Bendixson regions in blue.}
\label{casometodoHA2}
\end{figure}

Note that the cuts of the transversal ellipses with the horizontal axis give a very good approximation of the cuts of the limit cycle with the horizontal axis. Thus, we obtain by these transversal ellipses a good location of the limit cycles.

\section*{Acknowledgements}

The authors are partially supported by a MINECO/FEDER grant number MTM2011-22877 and by an AGAUR (Generalitat de Ca\-ta\-lu\-nya) grant number 2014SGR 1204.

\end{document}